\documentclass[12pt]{article}
\usepackage{amssymb, amscd, amsthm, amsmath,latexsym, amstext,mathrsfs,verbatim,longtable,array,multirow}
\usepackage[all]{xy}
\usepackage{graphicx,enumerate, url, makecell,longtable}
\usepackage{textcomp,booktabs}
\usepackage{geometry}
\def\lijntje{\vrule height2.4pt depth-2pt width0.5in}

\def\vlijntje{\vrule height0.45in depth0.4pt width0.4pt}
\def\vlijn{\buildrel {\hbox to 0pt{\hss$\textstyle\circ$\hss}}\over\vlijntje}

\def\dlijntje{{\vrule height2pt depth-1.6pt
		width0.5in}\llap{\vrule height4pt depth-3.6pt width0.5in}}
\def\tlijntje{{\vrule height1.7pt depth-1.3pt
		width0.5in}\llap{\vrule height3.0pt depth-2.6pt width0.5in}\llap{\vrule height4.3pt depth-3.9pt width0.5in}
}
\def\vtriple#1\over#2\over#3{\mathrel{\mathop{\kern0pt #2}\limits_{\hbox
			to 0pt{\hss$#1$\hss}}^{\hbox to 0pt{\hss$#3$\hss}}}}
\def\rvtriple#1\over#2\over#3{\mathrel{\mathop{\kern0pt #2}\limits_{\hbox
			to 0pt{\hss$#3$\hss}}^{\hbox to 0pt{\hss$#1$\hss}}}}
\def\Ct{\vtriple{\scriptstyle 2}\over\circ\over{}
	\kern-1pt\lijntje\kern-1pt\vtriple{\scriptstyle 1}\over\circ\over{}
	\kern-4pt{\dlijntje \kern -25pt<}\kern8pt
	\vtriple{\scriptstyle 0}\over\circ\over{}\kern-1pt
}
\def\Bt{\vtriple{\scriptstyle 2}\over\circ\over{}
	\kern-1pt\lijntje\kern-1pt\vtriple{\scriptstyle 1}\over\circ\over{}
	\kern-4pt{\dlijntje \kern -25pt>}\kern8pt
	\vtriple{\scriptstyle 0}\over\circ\over{}\kern-1pt}

\def\ddA{{\rm A}}
\def\ddD{{\rm D}}

\def\det{{\rm det}}

\def\ddB{{\rm B}}
\def\ddC{{\rm C}}
\def\ddD{{\rm D}}

\def\ddG{{\rm G}}

\newcommand{\C}{\mathbb C}

\newcommand{\N}{\mathbb N}

\newcommand{\Z}{\mathbb Z}

\def\Dm{\vtriple{\scriptstyle n+1}\over\circ\over{}\kern-1pt\lijntje\kern-1pt
	\vtriple{\scriptstyle{n}}\over\circ\over{}
	\cdots\cdots\vtriple{\scriptstyle 4}\over\circ\over{}\kern-1pt\lijntje\kern-1pt
	\vtriple{\scriptstyle 3}\over\circ\over{\buildrel
		{\scriptstyle 2}\over\vlijn}\kern-1pt\lijntje\kern-1pt
	\vtriple{1}\over\circ\over{}\kern-1pt}

\def\Dn{\vtriple{\scriptstyle n}\over\circ\over{}\kern-1pt\lijntje\kern-1pt
	\vtriple{\scriptstyle{n-1}}\over\circ\over{}
	\cdots\cdots\vtriple{\scriptstyle 4}\over\circ\over{}\kern-1pt\lijntje\kern-1pt
	\vtriple{\scriptstyle 3}\over\circ\over{\buildrel
		{\scriptstyle 2}\over\vlijn}\kern-1pt\lijntje\kern-1pt
	\vtriple{1}\over\circ\over{}\kern-1pt}
\def\En{\vtriple{\scriptstyle n}\over\circ\over{}\kern-1pt\lijntje\kern-1pt
	\vtriple{\scriptstyle{n-1}}\over\circ\over{}
	\cdots\cdots\vtriple{\scriptstyle 5}\over\circ\over{}\kern-1pt\lijntje\kern-1pt
	\vtriple{\scriptstyle 4}\over\circ\over{\buildrel
		{\scriptstyle 2}\over\vlijn}\kern-1pt\lijntje\kern-1pt
	\vtriple{\scriptstyle 3}\over\circ\over{}\kern-1pt\lijntje\kern-1pt
	\vtriple{\scriptstyle 1}\over\circ\over{}\kern-1pt}
\def\An{\vtriple{\scriptstyle n}\over\circ\over{}\kern-1pt\lijntje\kern-1pt
	\vtriple{\scriptstyle{n-1}}\over\circ\over{}\kern-1pt\lijntje\kern-1pt
	\vtriple{\scriptstyle n-2}\over\circ\over{}
	\cdots\cdots
	\vtriple{\scriptstyle 2}\over\circ\over{}\kern-1pt\lijntje\kern-1pt
	\vtriple{\scriptstyle 1}\over\circ\over{}\kern-1pt}
\def\Cn{\vtriple{\scriptstyle n-1}\over\circ\over{}
	\kern-1pt\lijntje\kern-1pt\vtriple{\scriptstyle{n-2}}\over\circ\over{}
	\cdots\cdots
	\vtriple{\scriptstyle 2}\over\circ\over{}
	\kern-1pt\lijntje\kern-1pt\vtriple{\scriptstyle 1}\over\circ\over{}
	\kern-4pt{\dlijntje \kern -25pt<}\kern10pt
	\vtriple{\scriptstyle 0}\over\circ\over{}\kern-1pt}
\def\Ct{\vtriple{\scriptstyle 2}\over\circ\over{}
	\kern-1pt\lijntje\kern-1pt\vtriple{\scriptstyle 1}\over\circ\over{}
	\kern-4pt{\dlijntje \kern -25pt<}\kern12pt
	\vtriple{\scriptstyle 0}\over\circ\over{}\kern-1pt
}
\def\Bn{\vtriple{\scriptstyle n-1}\over\circ\over{}
	\kern-1pt\lijntje\kern-1pt\vtriple{\scriptstyle{n-2}}\over\circ\over{}
	\cdots\cdots
	\vtriple{\scriptstyle 2}\over\circ\over{}
	\kern-1pt\lijntje\kern-1pt\vtriple{\scriptstyle 1}\over\circ\over{}
	\kern-4pt{\dlijntje \kern -25pt>}\kern10pt
	\vtriple{\scriptstyle 0}\over\circ\over{}\kern-1pt}
\def\Bt{\vtriple{\scriptstyle 2}\over\circ\over{}
	\kern-1pt\lijntje\kern-1pt\vtriple{\scriptstyle 1}\over\circ\over{}
	\kern-4pt{\dlijntje \kern -25pt>}\kern12pt
	\vtriple{\scriptstyle 0}\over\circ\over{}\kern-1pt}
\def\Es{\vtriple{\scriptstyle 6}\over\circ\over{}\kern-1pt\lijntje\kern-1pt
	\vtriple{\scriptstyle 5}\over\circ\over{}\kern-1pt\lijntje\kern-1pt
	\vtriple{\scriptstyle 4}\over\circ\over{\buildrel
		{\scriptstyle 2}\over\vlijn}\kern-1pt\lijntje\kern-1pt
	\vtriple{3}\over\circ\over{}\kern-1pt\lijntje\kern-1pt
	\vtriple{\scriptstyle 1}\over\circ\over{}\kern-1pt}
\def\Ff{
	\vtriple{\scriptstyle 1}\over\circ\over{}
	\kern-1pt\lijntje\kern-1pt\vtriple{\scriptstyle 2}\over\circ\over{}
	\kern-4pt{\dlijntje \kern -25pt<}\kern10pt
	\vtriple{\scriptstyle 3}\over\circ\over{}\kern-1pt\lijntje\kern-1pt
	\vtriple{\scriptstyle 4}\over\circ\over{}
	\kern-1pt}
\def\Ht{
	\vtriple{\scriptstyle 1}\over\circ\over{}
	\kern-1pt\overset{5}{\lijntje}\kern-1pt\vtriple{\scriptstyle 2}\over\circ\over{}
	\kern-1pt\lijntje\kern-1pt
	\vtriple{\scriptstyle 3}\over\circ\over{}\kern-1pt}
\def\Hf{
	\vtriple{\scriptstyle 1}\over\circ\over{}
	\kern-1pt\overset{5}{\lijntje}\kern-1pt\vtriple{\scriptstyle 2}\over\circ\over{}
	\kern-1pt\lijntje\kern-1pt
	\vtriple{\scriptstyle 3}\over\circ\over{}\kern-1pt\lijntje\kern-1pt
	\vtriple{\scriptstyle 4}\over\circ\over{}
	\kern-1pt}
\def\In{
	\vtriple{\scriptstyle 0}\over\circ\over{}
	\kern-1pt\overset{n}{\lijntje}\kern-1pt\vtriple{\scriptstyle 1}\over\circ\over{}
	\kern-1pt}
\def\Gt{
	\vtriple{\scriptstyle 0}\over\circ\over{}
	\kern-4pt{\tlijntje\kern -25pt<}\kern 10pt\vtriple{\scriptstyle 1}\over\circ\over{}
	\kern-1pt}
\def\EBn{\vtriple{\scriptstyle n-1}\over\circ\over{}
	\kern-1pt\lijntje\kern-1pt\vtriple{\scriptstyle{n-2}}\over\circ\over{\buildrel
		{\scriptstyle -1}\over\vlijn}\cdots\cdots
	\vtriple{\scriptstyle 2}\over\circ\over{}
	\kern-1pt\lijntje\kern-1pt\vtriple{\scriptstyle 1}\over\circ\over{}
	\kern-4pt{\dlijntje \kern -25pt<}\kern8pt
	\vtriple{\scriptstyle 0}\over\circ\over{}\kern-1pt}
\def\Cn{\vtriple{\scriptstyle n-1}\over\circ\over{}
	\kern-1pt\lijntje\kern-1pt\vtriple{\scriptstyle{n-2}}\over\circ\over{}
	\cdots\cdots
	\vtriple{\scriptstyle 2}\over\circ\over{}
	\kern-1pt\lijntje\kern-1pt\vtriple{\scriptstyle 1}\over\circ\over{}
	\kern-4pt{\dlijntje \kern -25pt<}\kern10pt
	\vtriple{\scriptstyle 0}\over\circ\over{}\kern-1pt}
\def\ECn{\vtriple{\scriptstyle -2}\over\circ\over{}
	\kern-4pt{\dlijntje \kern -25pt>}\kern8pt\vtriple{\scriptstyle n-1}\over\circ\over{}
	\kern-1pt\lijntje\kern-1pt\vtriple{\scriptstyle{n-2}}\over\circ\over{}
	\cdots\cdots
	\vtriple{\scriptstyle 2}\over\circ\over{}
	\kern-1pt\lijntje\kern-1pt\vtriple{\scriptstyle 1}\over\circ\over{}
	\kern-4pt{\dlijntje \kern -25pt<}\kern12pt
	\vtriple{\scriptstyle 0}\over\circ\over{}\kern-1pt}
\def\Fo{\vtriple{\scriptstyle -1}\over\circ\over{}
	\kern-1pt\lijntje\kern-1pt
	\vtriple{\scriptstyle 1}\over\circ\over{}
	\kern-1pt\lijntje\kern-1pt\vtriple{\scriptstyle 2}\over\circ\over{}
	\kern-4pt{\dlijntje \kern -25pt<}\kern8pt
	\vtriple{\scriptstyle 3}\over\circ\over{}\kern-1pt\lijntje\kern-1pt
	\vtriple{\scriptstyle 4}\over\circ\over{}
	\kern-1pt}
\def\Ft{
	\vtriple{\scriptstyle 1}\over\circ\over{}
	\kern-1pt\lijntje\kern-1pt\vtriple{\scriptstyle 2}\over\circ\over{}
	\kern-4pt{\dlijntje \kern -25pt<}\kern8pt
	\vtriple{\scriptstyle 3}\over\circ\over{}\kern-1pt\lijntje\kern-1pt
	\vtriple{\scriptstyle 4}\over\circ\over{}
	\kern-1pt\lijntje\kern-1pt
	\vtriple{\scriptstyle -2}\over\circ\over{}
	\kern-1pt}
\def\Go{\vtriple{\scriptstyle -1}\over\circ\over{}
	\kern-1pt\lijntje\kern-1pt
	\vtriple{\scriptstyle 0}\over\circ\over{}
	\kern-4pt{\tlijntje\kern -25pt<}\kern 12pt\vtriple{\scriptstyle 1}\over\circ\over{}
	\kern-1pt}
\def\Gf{
	\vtriple{\scriptstyle 0}\over\circ\over{}
	\kern-4pt{\tlijntje\kern -25pt<}\kern 12pt\vtriple{\scriptstyle 1}\over\circ\over{}
	\kern-1pt\lijntje\kern-1pt
	\vtriple{\scriptstyle -2}\over\circ\over{}
	\kern-1pt}

\numberwithin{equation}{section}

\newtheorem{lemma}{Lemma}[section]

\newtheorem{prop}[lemma]{Proposition}
\newtheorem{thm}[lemma]{Theorem}

\newtheorem{example}[lemma]{Example}
\theoremstyle{remark}
\newtheorem{rem}[lemma]{Remark}

\theoremstyle{definition}

\newtheorem{defn}[lemma]{Definition}
\newtheorem{conj}{Conjecture}[section]

\begin{document}
	\title{Characteristic polynomials for  classical Lie algebras }
	\author{ Chenyue Feng, Shoumin Liu\footnote{The author is  funded by the NSFC (Grant No. 11971181, Grant No.12271298) },   Xumin Wang}
	\date{}
	\maketitle

	\begin{abstract}
		In this paper,  we will compute the characteristic polynomials for finite dimensional representations of classical complex Lie algebras and the exceptional Lie algebra of type $\ddG_2$, which can be obtained through
		the orbits of  integral weights under the action of their corresponding Weyl groups  and the invariant polynomial theory of the Weyl groups. We show that
        the characteristic polynomials can be decomposed into products of irreducible orbital factors, each of which is invariant under the action of their corresponding Weyl groups.
		
	\end{abstract}

	\section{Introduction}

	\hspace{1.5em}Given a finite group $G = \{1, g_1, \cdots, g_n\}$ with the left regular representation $\lambda_G$ of $G$, Dedekind studied the determinant
	$$f_{\lambda_G}(z) = det(z_0I + z_1\lambda_G(g_1) + \cdots + z_n\lambda_G(g_n))$$
	for some groups and  its  factorizations.
	It is natural that  one may define the corresponding polynomial $f_{\pi}(z)$ associated to a general finite dimensional representation $\pi$ of $G$. Let $\hat{G}$ be the set of equivalence classes of irreducible unitary representations of $G$, and $d_{\pi}$ be the dimension of the representation $\pi$. Later in 1896 Frobenius proved the following theorem.
	
	\begin{thm}\label{0}
		If $G = \{1, g_1, \cdots, g_n\}$ is a finite group, then
		$$f_{\lambda_G}(z)=\prod_{\pi\in\hat{G}}(f_{\pi}(z))^{d_\pi}$$
		Moreover, each $f_{\pi}(z)$ is an irreducible polynomial of degree $d_{\pi}$.
	\end{thm}
	
The polynomial $f_{\pi}(z)$ was later called the group determinant of $G$ and many related
works by Dedekind and Frobenius are inspired by group representation theory.	
We refer to \cite{Cu2000,D1969,Di1902,Di1921,Di1975,F1896} for some illustrations of this topic.
However, the characteristic polynomials for several general matrices are a new frontier in linear algebra. In \cite{Y2009}, the notion of projective spectrum of operators is defined by  Yang through the multiparameter pencil, and  multivariable homogeneous characteristic polynomials have been studied there.
Fruitful results have been obtained in \cite{CMK2016, GY2017, GR2014, HY2023, HY2024,  HY2018, KY2021}.
For several matrices $A_1, \dots, A_n$ of equal size, their characteristic polynomial is defined as
	$$f_A(z)=det(z_0I+z_1A_1+\cdots+z_nA_n), \quad z=(z_0, \dots, z_n)\in \C^{n+1}.$$
	This generalization of characteristic polynomials contains a lot of information about this linear pencil of matrices. For example, the polynomials $f_A(z)$ is irreducible, then $A_1, \dots, A_n$ have no nontrivial common invariant subspace. Several normal matrices pairwisely commute if and only if $f_A(z)$ is a product of linear factors. Two pairs of projection matrices are unitarily equivalent if and only if their characteristic polynomials coincide, and so on. We refer the reader to the
	paper \cite{Y2024} for these and more conclusions. Without any doubt, linear pencils of matrices play an important role in a wide range of fields in matrix theory, operator theory, representation theory of groups and Lie algebras.
	
	It is natural to consider similar topics for finite dimensional Lie algebras.
	For a Lie algebra $\mathfrak{g}$ with a basis
	$\{x_1,\dots,x_n\}$, the characteristic polynomial of its adjoint representation
	$$f_{\mathfrak{g}}=\mathrm{det}(z_0I + z_{1}\mathrm{ad}x_1 + \cdots + z_{n}\mathrm{ad}x_n)$$ is investigated in \cite{KY2021}. It is shown that $f_\mathfrak{g}$ is invariant under the automorphism group $Aut(\mathfrak{g})$.
	Let $\phi:\mathfrak{sl}(2, \C)\rightarrow \mathfrak{gl}(V)$ be an irreducible representation of  $\mathfrak{sl}(2, \C)$ of dimension $m+1$. The characteristic polynomial
	\begin{eqnarray}\label{fpi}
		f_\phi(z_0,z_1,z_2,z_3)=\begin{cases}
			z_0\prod_{l=1}^{m/2}\left(z_0^2-4l^2(z_1^2+z_2z_3)\right) \quad 2\mid m\\
			\\
			\prod_{l=0}^{(m-1)/2}\left(z_0^2-(2l+1)^2(z_1^2+z_2z_3)\right)\quad 2\nmid m
		\end{cases}
	\end{eqnarray}
	is obtained in \cite{CCD2019, HZ2019, H2021}.
	When  the homomorphism $\phi$ is an arbitrary finite dimensional representation of  $\mathfrak{sl}(2, \C)$, we
	 denote the dimension of the eigenspace of $\phi(h)$ for the eigenvalue $n$ by $d_{n,\phi}$, $n\in \Z$, with \begin{equation*}
	 h=\left(\begin{array}{cc}
	 	1 & 0 \\
	 	0 & -1
	 \end{array}\right).
	 \end{equation*}
	
	 The characteristic polynomial
	\begin{equation}\label{FIRST}
		f_{\phi}(z_0,z_1,z_2,z_3)=z_0^{d_{0,\phi}}\prod_{n\geq 1}\left(z_0^2-n^2(z_1^2+z_2z_3)\right)^{d_{n,\phi}}
	\end{equation}
	is obtained in \cite{JL2022},
	where the authors prove that there is a correspondence between finite dimensional representations of $\mathfrak{sl}(2, \C)$ and their characteristic polynomials. Similar results hold for simple Lie algebra by the correspondence between characteristic polynomials and the linearized polynomials, meanwhile authors defined a monoid structure on these characteristic polynomials related to the tensor products of the representations. For more information, we refer the reader to \cite{GLW2024}.
	In \cite{KKSW}, authors gave a characterization of the linearization of characteristic polynomials which is independent of prescribed representations.
	
	In 2023, de Graaf put forward the following conjecture on generalised characteristic
	polynomials of  complex simple Lie algebras, which is abbreviated as the
	``\textbf{de Graaf's Conjecture}" formed in the below.
	\begin{conj}
	Let $\mu_1, \dots, \mu_s$ be the different dominant weights of irreducible $\mathfrak{g}$-module $V$ and $k_i$ be the size of the orbit of $\mu_i$ under the $\mathscr{W}$, $i=1, \dots, s$. Then
	$$f_{\phi}=\prod_{\mu_i\in\Pi(\lambda)}(f_{\mu_i})^{m_\lambda(\mu_i)},$$
	 where $f_{\mu_i}$ is an irreducible polynomial of degree $k_i=\lvert O_{\mu_i} \rvert$ obtained from the orbit of $\mu_i$,  $m_\lambda(\mu_{i})$ is the multiplicity of $\mu_{i}$ in $V$.
	\end{conj}
	 This conjecture will be partially proved in this paper mainly  for classical Lie algebras, 
	and the arrangement of this paper is as follows.
	Preliminary results are presented in Section 2 for preparations. In Section 3, we first demonstrate that the characteristic polynomials of the finite dimensional representations of $\mathfrak{sl}(3, \C)$ can be obtained from the orbits of integral weights under the action of the Weyl group of $\mathfrak{sl}(3, \C)$. Subsequently, this conclusion is extended to the characteristic polynomials of the finite dimensional representations of $\mathfrak{sl}(n, \C)$ in section 3. In Section 4, we further generalize this conclusion to the characteristic polynomials on the complex simple Lie algebra of type   $\ddB_{n}$, $\ddC_{n}$, $\ddD_{n}(n\geq 4)$ and $\ddG_2.$ Therefore, the characteristic polynomials of complex simple Lie algebras of the above types are described clearly.
	
	\textbf{Acknowledgments}
	
We would like to appreciate Prof. R.Yang  for Conveying  us the conjectures posed by
 Prof. W.de.Graaf  and many helpful discussions.

	\section{Some preliminary results}
	\hspace{1.3em}We first recall some results for Lie algebras which can be found in \cite{B2002, FH1991, H1972}. Let $\mathfrak{g}$ be a finite dimensional complex simple Lie algebra, $\Phi$
	be the root system of $\mathfrak{g}$ with Weyl group $\mathscr{W}$, and  $\Pi$=\{$\alpha_1$, $\dots$,$\alpha_n$\} be simple roots of $\Phi$. Let ${\mathfrak{h}}$ be a Cartan subalgebra of $\mathfrak{g}$, $\{h_{\alpha_1},\dots,h_{\alpha_n}\}$
	be a basis of ${\mathfrak{h}}$ corresponding to $\Pi$, and $E_{\alpha}$ be the root vector for  each $\alpha\in\Phi$. It is well known that  the set $\mathcal{A}$=$\{h_{\alpha_1},\dots,h_{\alpha_n}, E_{\alpha}(\alpha\in\Phi)\}$
	is a canonical basis of $\mathfrak{g}$, and we have the
	Cartan decomposition $$\mathfrak{g}={\mathfrak{h}}\oplus \sum_{\alpha\in\Phi} \C E_{\alpha}.$$
	Suppose $\phi:\mathfrak{g}\rightarrow \mathfrak{gl}(V)$ is a finite dimensional linear representation of $\mathfrak{g}$, with $V$ being a $\mathfrak{g}$-module.
	\begin{defn}
		The polynomial $$f_{\phi}(z_0,z_1,\dots,z_n,z_\alpha)=\mathrm{det}\left(z_{0}I+\sum_{i=1}^{n}\phi(h_{\alpha_i})z_i+\sum_{\alpha\in\Phi}\phi(E_{\alpha})z_{\alpha}\right)$$
		is called the characteristic polynomial of $\mathfrak{g}$ with respect to the basis $\mathcal{A}$ and the representation $\phi$,
		where $z_{0}$, $z_{1}$, $\dots$, $z_{n}$, $z_\alpha$($\alpha\in\Phi$) are indeterminants.
		Sometimes, let $L=\sum_{i=1}^{n}z_ih_{\alpha_i}+\sum_{\alpha\in\Phi}z_{\alpha}E_{\alpha}$,
		we write
		$$f_{\phi}(z_0,L)=f_{\phi}(z_0,z_1,\dots,z_n,z_\alpha)=\mathrm{det}(z_0I+L).$$
	\end{defn}
	\begin{rem}
		Here we recall the definition of the  formal characters $\Z[\Lambda]$ of a simple Lie algebra $\mathfrak{g}$ from \cite[Section 22.5]{H1972}.
		Let $\Lambda(\mathfrak{g})\subseteq\mathfrak{h}^*$ be the set of integral weight lattice, namely all $\lambda \in \mathfrak{h}^*$ which $\left<\lambda, \alpha\right>\in \Z$($\alpha\in \Phi$).
		 $\Lambda(\mathfrak{g})$ has an integral basis $\{\beta_i\}$ and dominant weight $\mu \in \Lambda(\mathfrak{g})$ can only be written in the following form with respect to the integral basis,
		 $$\mu=\sum_il_i\beta_i$$ with $l_i\in \N$.
		Furthermore we have
		\begin{eqnarray}\label{2.1}
			\left<\beta_i, \alpha_j\right>=\frac{2(\beta_i,\alpha_j)}{(\alpha_j,\alpha_j)}=\beta_i(h_{\alpha_j})=\delta_{ij},\ \
		\gamma_i(\beta_j)=\beta_j-\delta_{ij}\alpha_i	
		\end{eqnarray}
		where $\alpha_j$ is simple root of $\Phi$ and $\gamma_i \in \mathscr{W}$ is the reflection corresponding to $\alpha_i$.
		All of these results can be seen in \cite[Chapter 13]{H1972}.
		Define the stabilizer of $\mu$ under the reflection action of $\mathscr{W}$ as follows
		$$stab(\mu)=\{\gamma \in \mathscr{W} \mid \gamma(\mu)=\mu\}.$$
		Let $\Z[\Lambda]$ be the free abelian group  with $\{e(\eta)\}_{\eta \in \Lambda}$ being its basis.
		From \cite[Section 22.5]{H1972},  the abelian group  $\Z[\Lambda]$ has a commutative ring structure through decreeing that $e(\lambda)e(\eta)=e(\lambda+\eta)$.
		For $\lambda\in\Lambda^{+}$, we suppose that $V(\lambda)$ is an irreducible finite dimensional module of $\mathfrak{g}$ with the highest weight $\lambda$ and
		$\Pi(\lambda)$ denote the set of its weights.
		We define the formal character for $V(\lambda)$ as \begin{equation}\label{form2}
			ch_{\lambda}=\sum_{\eta\in \Pi(\lambda)}m_\lambda(\eta)e(\eta),
		\end{equation}
		where $m_\lambda(\eta)$ is the multiplicity of $\eta$ in $V(\lambda)$, namely the dimension of the eigenspace corresponding to the weight $\eta$. It is known that the
		$m_\lambda(\eta)$s are same for the $\eta$s on the same orbits of the Weyl groups associated to the simple Lie algebra.
		For a general finite dimension representation $V$, if $V(\lambda)=V(\lambda_1)\oplus \cdots V(\lambda_t)$, we define its formal character as $$ch_{V}=\sum_{i=1}^{t}ch_{\lambda_i}.$$
	\end{rem}

	\section{Characteristic polynomials for  $\mathfrak{sl}(n, \C)$ and $\mathfrak{sl}(4, \C)$}\label{sl3}
	\subsection{Cases for the Lie algebra  $\mathfrak{sl}(3, \C)$}
	\hspace{1.3em}Now we focus on the Lie algebra $\mathfrak{sl}(3, \C)$.  Suppose that $E_{i,j}$ be the $3\times 3$ matrix  with $1$ at the entry $(i,j)$, others being $0$.
	Then we know that the  Lie algebra  $\mathfrak{sl}(3, \C)$ has the Cartan decomposition with basis $H_{1}=E_{1,1}-E_{2,2}$, $H_{2}=E_{2,2}-E_{3,3}$,
	and $E_{i,j}$ for $1\leq i\neq j\leq 3$. It is  known that the Lie algebra $\mathfrak{sl}(3, \C)$ is the simple Lie algebras of type $\ddA_2$.
	Let $\mathscr{W}(\ddA_2)$, $\Phi(\ddA_2)$, $\Lambda(\ddA_2)$ be the Weyl group, root system and integral weight lattice associted to  $\mathfrak{sl}(3, \C)$, respectively.
	Suppose that the simple roots of $\Phi(\ddA_2)$ are $\alpha_1$ and $\alpha_2$, then $\Lambda(\ddA_2)$ has the integral basis $$\beta_1=\frac{2\alpha_1+\alpha_2}{3}, \quad \beta_2=\frac{\alpha_1+2\alpha_2}{3}.$$
	\begin{lemma}\label{lem3.1}
		Under the reflection action of  $\mathscr{W}(\ddA_2)$ on $\Phi(\ddA_2)$ , the possible numbers of weights   in each orbit  are $1$, $3$, and $6$.
		Furthermore,  each orbit contains a unique element being a dominant weight  of the form $l_1\beta_1+l_2\beta_2$, with $l_1$, $l_2\in \N$.
	\end{lemma}
	\begin{proof} From \cite{B2002}, for any $\beta \in \Lambda(\ddA_2)$ and  $\alpha\in \Phi(\ddA_2)$, the stabilizer of $\beta$ in $\mathscr{W}(\ddA_2)$, is generated by the reflections $r_{\alpha}$, with $\alpha\perp \beta$, and isomorphic to a parabolic subgroup of
		$\mathscr{W}(\ddA_2)$, which is possibly being  the trivial subgroup, $\mathscr{W}(\ddA_1)$ or $\mathscr{W}(\ddA_2)$. Therefore by the Lagrange's Theorem, the possible numbers of weights   in each orbit  are $1$, $3$, and $6$. by \cite[Section13.2, Lemma A]{H1972},
		we see that each orbit has a unique element in the form $l_1\beta_1+l_2\beta_2$, with $l_1$, $l_2\in \N$. and we can compute it and obtain that
		\begin{eqnarray}\#\left(\mathscr{W}(\ddA_2)(l_1\beta_1+l_2\beta_2)\right)=\begin{cases}
				6, \quad l_1l_2\neq 0
				\\
				3, \quad  \text{only one of } l_1 \text{ and } l_2 \text{ being 0}
				\\
				1, \quad l_1=l_2=0
			\end{cases}
		\end{eqnarray}
		\end{proof}
		\begin{lemma}\label{stabistor}
			Suppose $\mu=\sum_{k=1}^{n-1}l_k\beta_k$ and $\alpha=\sum_{i=1}^{n-1}t_i\alpha_i \in \Phi^+$ with $l_k \geq 0$ and $t_i \geq 0$. Then 	$stab(\mu)=\{\gamma_j \in \mathscr{W}(\ddA_{n-1})  \mid l_j=0 \}.$
		\end{lemma}
		\begin{proof}
			From (\ref{2.1}), the subgroup $stab(\mu)$ is generated by those roots perpendicular to $\mu$, then it follows as lamma (\ref{lem3.1}).
		\end{proof}	
	Suppose that $V(\lambda)$ is an irreducible finite dimensional module of $\mathfrak{sl}(3, \C)$ with the highest weight $\lambda$ and
	$\Pi(\lambda)$ denote the set of its weights.
	Let $\mu=l_1\beta_1+l_2\beta_2$,
	$O_{\mu}=\{\gamma(\mu) \mid \gamma\in\mathscr{W}(\ddA_2) \}$, $\lvert O_{\mu} \rvert$ be the number of elements in $O_{\mu}$,
	and $orb(V(\lambda))=\{O_{\mu} \mid O_{\mu}\cap \Pi(\lambda)\neq \varnothing\}$.
	Then we have the following lemma.
	
	\begin{lemma}\label{lem3.2}
		It follows that
		\begin{eqnarray}
			O_{\mu}=\begin{cases}
				\{l_1\beta_1, l_1(\beta_2-\beta_1),-l_1\beta_2\}, \quad l_1\neq 0, \quad l_2=0\\
				\\
				\{l_2\beta_2, l_2(\beta_1-\beta_2),-l_2\beta_1\}, \quad l_2\neq 0, \quad l_1=0\\
				\\
				\{l_1\beta_1+l_2\beta_2,\, -l_1\beta_1+(l_1+l_2)\beta_2,\,l_2\beta_1-(l_1+l_2)\beta_2\\
				-l_2\beta_1-l_1\beta_2,\, (l_1+l_2)\beta_1-l_2\beta_2,\, -(l_1+l_2)\beta_1+l_1\beta_2\},\quad l_1l_2\neq 0\\
				\\
				\{0\}, \quad l_1=l_2=0
			\end{cases}
		\end{eqnarray}
	\end{lemma}
	\begin{proof} Let $\mathscr{W}(\ddA_2)=\left<\gamma_1, \gamma_2\right>$, with $\gamma_1$ and $\gamma_2$ being the reflections corresponding to $\alpha_1$ and $\alpha_2$, respectively.
		From (\ref{2.1}) we can compute this directly.
	\end{proof}
	
	In this paper, we use $``\sim"$ to indicate that the two matrices are similar.
	\begin{thm}\label{1}
		Let $\mu_1, \dots, \mu_s$ be the different dominant weights of irreducible $\mathfrak{sl}(3, \C)$-module $V(\lambda)$ and $k_i$ be the size of the orbit of $\mu_i$ under the $\mathscr{W}(\ddA_2)$, $i=1, \dots, s$. Then
		$$f_{\phi}(z_0,L)=\prod_{\mu_i\in\Pi(\lambda)}(f_{\mu_i})^{m_\lambda(\mu_i)},$$
		 where $f_{\mu_i}$ is an irreducible polynomial of degree $k_i=\lvert O_{\mu_i} \rvert$ obtained from the orbit of $\mu_i$.
	\end{thm}
	
	\begin{proof}
		Suppose $$L=z_{11}H_1+z_{22}H_2+\sum_{1\leq i\neq j\leq 3}z_{ij}E_{ij},$$
		then we have
		\begin{equation}
			L=\left(\begin{array}{ccc}
				z_{11} & z_{12} & z_{13} \\
				z_{21} & z_{22}-z_{11} & z_{23}\\
				z_{31} & z_{32} & -z_{22}
			\end{array}\right),
		\end{equation}
		and the characteristic polynomial
		\begin{equation*}
			\begin{aligned}
				\mathrm{det}(z_0 I-L)=z_0^3+(-z_{11}^2+z_{11}z_{22}-z_{22}^2-z_{12}z_{21}-z_{13}z_{31}-z_{23}z_{32})z_0-\mathrm{det}(L).
			\end{aligned}
		\end{equation*}
		Let
		\begin{alignat*}{2}
			p&=-(z_{11}^2+z_{22}^2-z_{11}z_{22}+z_{12}z_{21}+z_{13}z_{31}+z_{23}z_{32}),\\
			q&=\mathrm{det}(L),
		\end{alignat*}
		then
		\begin{equation*}\label{form1}
			\mathrm{det}(z_0 I-L)=z_0^3+pz_0-q
		\end{equation*}
		with eigenvalues $\theta_1$, $\theta_2$, $\theta_3$.

		The elements in $\mathfrak{sl}(3, \C)$ that can be diagonalized are dominant in $\mathfrak{sl}(3, \C)$ being endowed with the Zariski topology, so we can just prove the theorem for diagonalizable elements in $\mathfrak{sl}(3, \C)$.
		Since $Tr(L)=0$, 
		$L$ is similar to the following diagonal matrix
		\begin{equation*}
			\left(\begin{array}{ccc}
				\theta_1 & 0 & 0 \\
				0 & -\theta_1-\theta_2 & 0\\
				0 & 0 & \theta_2
			\end{array}\right),
		\end{equation*}
		which implies $L\sim \theta_1H_1-\theta_2H_2=diag(\theta_1$, $\theta_3$, $\theta_2)$.
		Therefore,
		\begin{equation*}
			f_{\phi}(z_0,L)=\mathrm{det}\big(z_{0}I+\phi(\theta_1H_1-\theta_2H_2)\big).
		\end{equation*}

		The $\Pi(\lambda)$ is reclassified by orbits, then
		\begin{equation*}
			ch_{\lambda}=\sum_{O_{\mu}\subset O(\lambda)} \sum_{\gamma(\mu) \in O_{\mu}}m_\lambda(\mu)e(\mu),
		\end{equation*}	
		Since the weight on the same orbits have the same multiplicity,
		then we have
		\begin{eqnarray*}
			f_{\phi}(z_0,L)&=&\prod_{O_{\mu}\subset orb(V(\lambda))}\prod_{\gamma(\mu)\in O_{\mu}} \big(z_0+\theta_1 \gamma(\mu)(H_1)-\theta_2 \gamma(\mu)(H_2) \big)^{m_\lambda(\mu)}\\
			&=&\prod_{O_{\mu}\subset orb(V(\lambda))}\bigg(\prod_{\gamma \in \mathfrak{S}_n/stab(\mu)} \big(z_0+\theta_1 \gamma(\mu)(H_1)-\theta_2 \gamma(\mu)(H_2) \big)\bigg)^{m_\lambda(\mu)}.
		\end{eqnarray*}
	Let \begin{equation*}
		f_{\mu}=\prod_{\gamma \in \mathfrak{S}_n/stab(\mu)} \big(z_0+\theta_1 \gamma(\mu)(H_1)-\theta_2 \gamma(\mu)(H_2) \big),
	\end{equation*}
	where $\mu=l_1\beta_1+l_2\beta_2$.
	
	By the Vieta's theorem, we have
	$$p=\theta_1\theta_2+\theta_1\theta_3+\theta_2\theta_3,$$
	$$q=\theta_1\theta_2\theta_3.$$
	Hence, we divide it into four cases.
	
	\textbf{Case 1:}

	When $\mu=0$, $f_{\mu}=z_0.$
	
	\textbf{Case 2:}
	
	When $\mu=l_2\beta_2,$	$\phi(-\theta_2H_2)=diag(-\theta_2l_2, \theta_1l_2+\theta_2l_2, -\theta_1l_2),$
	then
	\begin{eqnarray*}
		f_{\mu}&=&\prod_{\gamma \in \mathfrak{S}_n/stab(\mu)} \big(z_0+\theta_1 \gamma(\mu)(H_1)-\theta_2 \gamma(\mu)(H_2) \big)\\
		&=&z_{0}^{3}-pl_2^2z_{0}-ql_2^3 \\
		&=&z_{0}^{3}+(z_{11}^2+z_{22}^2-z_{11}z_{22}+z_{12}z_{21}+z_{13}z_{31}+z_{23}z_{32})l_2^2z_{0}-\mathrm{det}(L)l_2^3.
	\end{eqnarray*}
	\textbf{Case 3:}
	
	When $\mu=l_1\beta_1$ similar to the case above,
	\begin{eqnarray*}
		f_{\mu}&=&z_{0}^{3}-pl_1^2z_{0}+ql_1^3 \\
		&=&z_{0}^{3}+(z_{11}^2+z_{22}^2-z_{11}z_{22}+z_{12}z_{21}+z_{13}z_{31}+z_{23}z_{32})l_1^2z_{0}+\mathrm{det}(L)l_1^3.
	\end{eqnarray*}
	\textbf{Case 4:}
	
	When $\mu=l_1\beta_1+ l_2\beta_2,$
	\begin{eqnarray*}
		&&\phi(\theta_1H_1-\theta_2H_2)\\
		&=&diag(\theta_1l_1-\theta_2l_2, \theta_3l_1-\theta_2l_2, \theta_2l_1-\theta_3l_2, \theta_2l_1-\theta_1l_2,\theta_1l_1-\theta_3l_2 , \theta_3l_1-\theta_1l_2),
	\end{eqnarray*}
	then
	\begin{eqnarray*}
		f_{\mu}&=&\prod_{\gamma \in \mathfrak{S}_n/stab(\mu)} \big(z_0+\theta_1 \gamma(\mu)(H_1)-\theta_2 \gamma(\mu)(H_2) \big)\\
		&=&z_{0}^{6}+\sum_{i=0}^{4}q_{l_1,l_2}^iz_{0}^i,
	\end{eqnarray*}
	where
	\begin{alignat*}{5}
		q_{l_1,l_2}^4&=2p(l_2^2+l_1^2+l_2l_1),\\
		q_{l_1,l_2}^3&=q(l_1-l_2)(2l_1+l_2)(l_1+2l_2),\\
		q_{l_1,l_2}^2&=p^2(l_2^2+l_1^2+l_2l_1)^2,\\
		q_{l_1,l_2}^1&=pq(l_1-l_2)(2l_1+l_2)(l_1+2l_2)(l_1^2+l_1 l_2+l_2^2),\\
		q_{l_1,l_2}^0&=
		q^2(l_1^2+l_1 l_2+l_2^2)^3+p^3(l_1^2 l_2^2(l_1+l_2)^2).
	\end{alignat*}
	Therefore, $f_\phi$ can be expressed as the product of orbital factors $f_\mu$, namely
	$$f_{\phi}(z_0,L)=(f_{\mu_1})^{m_\lambda(\mu_1)}\cdots (f_{\mu_s})^{m_\lambda(\mu_s)},$$
	where $\mu_i$ are dominant weights of irreducible $\mathfrak{sl}(3, \C)$-module $V$ and the degree of $f_{\mu_i}$ equal to the size of the orbit of $\mu_i$ under the $\mathscr{W}(\ddA_2)$.	

    For the irreducibility of $f_{\mu}$, we leave it to a general proof in Theorem \ref{3.10}.
\end{proof}

Throughout this paper, we call the $f_{\mu}$ the orbital factors.

\begin{example}
Suppose $\phi$ is the representation of $\mathfrak{sl}(3, \C)$ of dimension $15$ with highest weight $4\beta_2$ and dominant weights $\beta_2$, $2\beta_1$, $\beta_1+2\beta_2$. This can be referred to \cite[Lecture 12]{FH1991}. Let
$\mu_1=4\beta_2$, $\mu_2=\beta_1+2\beta_2$, $\mu_3=2\beta_1$, $\mu_4=\beta_2$, we have
\begin{alignat*}{4}
	f_{\mu_1}&=z_{0}^{3}-16pz_{0}-64q,\\
	f_{\mu_2}&=z_0^6+14pz_0^4-20qz_0^3+49p^2z_0^2-140pqz_0+343q^2+36p^3,\\
	f_{\mu_3}&=z_0^3-4pz_0+8q,\\
	f_{\mu_4}&=z_0^3-pz_0-q.
\end{alignat*}
After verification, the result of directly calculating $f_{\phi}(z_0,L)$ is indeed equal to $f_{\mu_1}f_{\mu_2}f_{\mu_3}f_{\mu_4}.$
\end{example}

\subsection{Characteristic polynomials for  $\mathfrak{sl}(n, \C)$} \label{An}
\hspace{1.3em}Let $E_{i,j}$ be the $n \times n$ matrix  with $1$ at the entry $(i,j)$, others being $0$. Suppose that
$\{H_i\}$ is the Chevalley basis for
the Cartan subalgebra  $\mathfrak{h}$ of $\mathfrak{sl}(n, \C)$,
with $H_i=E_{i,i}-E_{i+1,i+1} (1\leq i\leq n-1)$.
$\forall L\in\mathfrak{sl}(n, \C)$, we have
\begin{eqnarray} \label{indeterminant}
	L=\sum_{i=1}^{n-1}z_{i,i}H_i+\sum_{1\leq i\neq j\leq n}z_{i,j}E_{i,j}.
\end{eqnarray}

\begin{rem}\label{3.4}
	Suppose  $L_0 \in \mathfrak{sl}(n, \C)$
	has pairwise distinct eigenvalues
	$\theta_1(L_0), \dots, \theta_n(L_0)$ ordered by their length and argument, we can order these roots, namely $\lvert \theta_i(L_0) \rvert < \lvert \theta_j(L_0) \rvert$ or $\arg \theta_i(L_0)<\arg \theta_j(L_0)$ when
	$\lvert \theta_i(L_0) \rvert = \lvert \theta_j(L_0) \rvert$ with $i<j.$
	Therefore, there is an open neighbourhood $U$ of $L_0$ in $\mathfrak{g}$ such that any element $L $ in $U$ has distinct eigenvalues $\theta_1(L), \dots, \theta_n(L)$,
	those $L_0$ are dense under the usual complex topology and dominant under the Zariski's topology,
	and the ordered eigenvalues are continous with respect to the usual complex topology on $U$.
\end{rem}


\begin{thm}\label{3.5}
	Let $\mathfrak{h}$ be the Cartan subalgebra of classical Lie algebra $\mathfrak{sl}(n, \C)$ with the canonical basis $\{H_i\}_{i=1}^{n-1}$ and $L\in\mathfrak{sl}(n, \C)$ can be diagonalizable with distinct eigenvalues $\theta_1,\dots,\theta_{n}$, then
	$$L\sim \sum_{i=1}^{n-1}\bigg(\sum_{j=1}^{i}\theta_j\bigg)H_i.$$
\end{thm}
\begin{proof}
	Since $L\sim diag(\theta_1,\dots,\theta_n)$
	with $\sum_{i=1}^n\theta_i=0$, then by computation, it follows that
	$$L\sim \sum_{i=1}^{n-1}\bigg(\sum_{j=1}^{i}\theta_j\bigg)H_i.$$
\end{proof}

Let $\mu=\sum_{k=1}^{n-1}l_k\beta_k \in \Lambda(\ddA_{n-1})$ be a dominant weight where $\{\beta_k\}$ is a integra basis of integral weight lattice $\Lambda(\ddA_{n-1})$ and $l_k\in \N$.
 Considering $l_i$ being 0 or not, there are $2^{n-1}$ kinds of orbitals for the Weyl group, because of the permutations that take zero for $l_i$.
 By Lemma \ref{lem3.2}, similar to $\mathfrak{sl}(3, \C)$, we obtain orbit
\begin{eqnarray*}
	O_{\mu}=\{\gamma(\mu)\mid \gamma\in\mathscr{W}(\ddA_{n-1})\}.
\end{eqnarray*}
Define the stabilizer of $\mu$ as follows
$$stab(\mu)=\{\gamma \in \mathscr{W}(\ddA_{n-1}) \mid \gamma(\mu)=\mu\}.$$
\begin{lemma}\label{stabistor}
	Suppose $\mu=\sum_{k=1}^{n-1}l_k\beta_k$ and $\alpha=\sum_{i=1}^{n-1}t_i\alpha_i \in \Phi^+$ with $l_k \geq 0$ and $t_i \geq 0$. Then 	$stab(\mu)=\{\gamma_j \in \mathscr{W}(\ddA_{n-1})  \mid l_j=0 \}$
\end{lemma}
\begin{proof}
	From (\ref{2.1}), $stab(\mu)$ is generated by those roots perpendicular to $\mu$, then it follows as lamma (\ref{lem3.1}).
\end{proof}	
Since $\gamma(\mu) \in \mathfrak{h}^*$, and from (\ref{2.1})  we can obtain the values for $\gamma(\mu)(H_i)$ where $\gamma$ runs all over $\mathscr{W}(\ddA_{n-1})$, $1\leq i \leq n-1$.
Let $f_{\phi}(z_0,L)$ be the characteristic polynomial of them with respect to the representation $\phi$, we know that $f_{\phi}(z_0,L) \in \mathbb{C}[z_{1,1},\dots, z_{n-1, n-1},z_{i,j}].$
From Theorem \ref{3.5}, we have
\begin{eqnarray*}
	f_{\phi}(z_0,L)&=&\det\left(z_0I+\phi\bigg(\sum_{i=1}^{n-1}\bigg(\sum_{j=1}^{i}\theta_j\bigg)H_i\bigg)\right)\\
	&=&\det\left(z_0I+\sum_{i=1}^{n-1}\bigg(\sum_{j=1}^{i}\theta_j\bigg)\phi(H_i)\right)\\
	&=&\prod_{O_{\mu}\subset orb(V(\lambda))} (f_{\mu})^{m_\lambda(\mu)}
\end{eqnarray*}	
with orbital factor,
\begin{eqnarray*}
	f_{\mu}&=&\prod_{\gamma(\mu)\in O_{\mu}}\left(z_0I+\sum_{i=1}^{n-1}\bigg(\sum_{j=1}^{i}\theta_j\bigg)\gamma(\mu)(H_i)\right)\\
	&=&\prod_{\gamma \in \mathfrak{S}_n/stab(\mu)}\left(z_0I+\sum_{i=1}^{n-1}\bigg(\sum_{j=1}^{i}\theta_j\bigg)\gamma(\mu)(H_i)\right).
\end{eqnarray*}

\begin{rem}
	It's known that the symmetric group $\mathfrak{S}_n\simeq\mathscr{W}(\ddA_{n-1})$ can act on the polynomial ring $\mathbb{C}[\theta_1,\dots,\theta_{n}]$ by
	$$(\gamma \cdot f)(\theta_1,\dots,\theta_{n})=f(\theta_{\gamma(1)},\dots,\theta_{\gamma(n)})$$
	for $\gamma \in \mathfrak{S}_n$ and $f\in\mathbb{C}[\theta_1,\dots,\theta_{n}]$ where $\theta_1,\dots,\theta_{n}$ is regarded as indeterminants. The  representation of $\mathfrak{S}_n$ on $V$ as permutation matrices, with $V=\mathbb{C}^n \simeq \mathbb{C}\theta_1\oplus \cdots \oplus \mathbb{C}\theta_n.$
	Let $\{\theta_i\}_{i=1}^n$ be a basis of $V$, $W=\mathbb{C}(\theta_2-\theta_1)\oplus \cdots \oplus\mathbb{C}(\theta_n-\theta_1)$ and $W'=\mathbb{C}(\theta_1+\cdots+\theta_n)$.
	As the $\mathfrak{S}_n$-group modules, we have $V=W\oplus W'$, with $W$ being the canonical reflection representation for $\mathfrak{S}_n\simeq\mathscr{W}(\ddA_{n-1})$.

	Recall the elementary symmetric functions
	$$s_i(\theta_1,\dots,\theta_{n})=\sum_{1\leq j_1 < \cdots < j_i \leq n}\theta_{j_1} \cdots \theta_{j_i},$$
	and set $s_0=1$. 
	Define the subring $\mathbb{C}[\theta_1,\dots,\theta_{n}]^{\mathfrak{S}_n}$ of invariant polynomials,
	$$\mathbb{C}[\theta_1,\dots,\theta_{n}]^{\mathfrak{S}_n}=\{f\in\mathbb{C}[\theta_1,\dots,\theta_{n}]\mid \gamma \cdot f=f\ for\ all\ \gamma\in\mathfrak{S}_n\}.$$
	Clearly, $s_i \in \mathbb{C}[\theta_1,\dots,\theta_{n}]^{\mathfrak{S}_n}$ and $\mathbb{C}[\theta_1,\dots,\theta_{n}]^{\mathfrak{S}_n}=\mathbb{C}[s_1,\dots, s_n]$ by \cite[Section 16-2]{RK2001}.
	And we have the identity
	$$\det(z_0 I-L)=\sum_{j=0}^n(-1)^js_jz_0^{n-j},$$
	with $L \in \mathfrak{sl}(n,\C)$ which has eigenvalues $\theta_1,\dots,\theta_n$.
	Let $\mathbb{C}[z_{1,1},\dots, z_{n-1, n-1}, \dots, z_{i,j}, \dots]$ be the set of all Characteristic polynomials of $L$ in the (\ref{indeterminant}).
	From the Vieta's Theorem, $s_i \in \mathbb{C}[z_{1,1},\dots, z_{n-1, n-1},\dots, z_{i,j}, \dots].$
\end{rem}

\begin{thm}\label{2}
	Extending the $\mathscr{W}(\ddA_{n-1})\simeq\mathfrak{S}_n$ actions on $z_{1,1},\dots, z_{n-1, n-1},z_{i,j}$ to $z_0$ by trivial action.
	The orbit factor $f_{\mu}$ for $\mathfrak{sl}(n, \C)$ is a  $\mathscr{W}(\ddA_{n-1})$ invariant polynomial, 
	so it follows that $f_{\mu} \in \mathbb{C}[s_1,\dots, s_n].$
	Furthermore, given that $f_\mu$ is a polynomial of $z_0$, the coefficients of $f_\mu$ can be expressed by polynomials of $s_1,\dots, s_n$.
	And then $f_{\mu}$ is a polynomial factor of $f_\phi$.
\end{thm}

\begin{proof}
	We take $\sigma, \gamma\in\mathscr{W}(\ddA_{n-1})$, then
	\begin{eqnarray*}
		\sigma(f_{\mu})&=&\sigma\left(\prod_{\gamma \in \mathfrak{S}_n/stab(\mu)}\left(z_0+\gamma  (\mu)\bigg(\sum_{i=1}^{n-1}\bigg(\sum_{j=1}^{i}\theta_j\bigg)H_i\bigg)\right)\right)\\
		&=&\prod_{\gamma \in \mathfrak{S}_n/stab(\mu)}\left(z_0+\sigma(\gamma  (\mu))\bigg(\sum_{i=1}^{n-1}\bigg(\sum_{j=1}^{i}\theta_j\bigg)H_i\bigg)\right).
	\end{eqnarray*}	
	It is well know that $\sigma \gamma$, $\gamma$ running over all representations of $\mathfrak{S}_n/stab(\mu)$, is another presentation of $\mathfrak{S}_n/stab(\mu)$.
	So 	$\sigma(f_{\mu})=f_{\mu}$.
\end{proof}

\begin{rem}
	Another way to prove the theorem \ref{2} is as follows, where $f_{\mu}$ is rewritten.
	
	In fact, we have
	\begin{eqnarray*}
		\gamma (\mu) \left(\sum_{i=1}^{n-1}\bigg(\sum_{j=1}^{i}\theta_j\bigg)H_i\right)&=&\left(\sum_{k=1}^{n-1}l_k\gamma(\beta_k)\right)  \left(\sum_{i=1}^{n-1}\bigg(\sum_{j=1}^{i}\theta_{j}\bigg)H_i\right)\\
		&=&\mu \left(\gamma^{-1}\bigg(\sum_{i=1}^{n-1}\bigg(\sum_{j=1}^{i}\theta_{j}\bigg)H_i\bigg)\right)\\
		&=&\mu\left(\gamma^{-1}(diag(\theta_1, \theta_2, \cdots, \theta_{n})) \right) \\
		&=&\mu\left(diag(\theta_{\gamma(1)}, \theta_{\gamma(2)}, \cdots, \theta_{\gamma(n)}) \right) \\
		&=&\sum_{i=1}^{n-1}\bigg(\sum_{j=1}^{i}\theta_{\gamma(j)}\bigg)\mu(H_i),\\
	\end{eqnarray*}
	with $\gamma \in \mathscr{W}(\ddA_{n-1})$.
	Then we have
	\begin{eqnarray*}
		f_{\mu}&=&\prod_{\gamma \in \mathfrak{S}_n/stab(\mu)}\left(z_0+\gamma  (\mu)\bigg(\sum_{i=1}^{n-1}\bigg(\sum_{j=1}^{i}\theta_j\bigg)H_i\bigg)\right)\\
		&=&\prod_{\gamma \in \mathfrak{S}_n/stab(\mu)}\left(z_0+\sum_{i=1}^{n-1}\bigg(\sum_{j=1}^{i}\theta_{\gamma(j)}\bigg)\mu(H_i)\right) \\
		&=&\prod_{\gamma \in \mathfrak{S}_n/stab(\mu)}\left(z_0+\sum_{i=1}^{n-1}\bigg(\sum_{j=1}^{i}\theta_{\gamma(j)}\bigg)l_i\right).
	\end{eqnarray*}	
	So we know that $f_{\mu}$ is a symmetric polynomial,
	therefore a $\mathfrak{S}_n$ invariant polynomial and can be written uniquely as an algebraic expression in $\{s_1, \dots, s_{n}\}$ from \cite[Section 16-2]{RK2001} or \cite[Chapter 5]{GW2009}.
\end{rem}

The above conclusion brings us to the following theorem.
\begin{thm}\label{3.10}\label{irreducible}
	Let $\mathfrak{g} \in \mathfrak{sl}(n, \C)$ with a Chevalley basis, $\mu_1, \dots, \mu_s$ be the different dominant weights of $\mathfrak{g}$-module $V$ and $k_i$ be the size of the orbit of $\mu_i$ under the Weyl group  $\mathscr{W}$, $i=1, \dots, s$.
	Then the characteristic polynomial
	$$f_{\phi}(z_0,L)=(f_{\mu_1})^{m_\lambda(\mu_1)}\cdots (f_{\mu_s})^{m_\lambda(\mu_s)},$$
	with $f_{\mu_i}$ irreducible of degree $k_i=\lvert O_{\mu_i} \rvert$.
\end{thm}

\begin{proof}
Here we just need to prove the irreducibility of $f_{\mu}$.
Let $F=\mathbb{C}[z_0, z_{1,1},\dots, z_{n-1, n-1},\dots, $ $z_{i,j}, \dots]$ be the quotient field
of $\mathbb{C}[z_0]$,
$K$ be the splitting field of $\mathrm{det}(z_0 I-L)$ where $L$ is in formula (\ref{indeterminant}). Then $Gal(K/ F)=\mathfrak{S}_n$. Consider each $\mu=\sum_{k=1}^{n-1}l_k\beta_k \in \Lambda(\ddA_{n-1})$, the orbit of $\mu$ under $\mathfrak{S}_n$ is corresponding to $\mathfrak{S}_n/stab(\mu)$.
Furthermore, let $K_\mu$ be the splitting field of $f_\mu$, by Galois theory, $Gal(K/K_\mu)=stab(\mu)$.
Therefore by \cite[Theorem 6.4.1]{DAC2012}, the Galois group $G$ for $f_{\mu}$  is isomorphic to $\mathfrak{S}_{n!}$ and then $G$ is transitive. In \cite[Proposition 6.3.7]{DAC2012}, we see that $f_{\mu}$ is irreducible.
\end{proof}

The following proposition show a method to compute each $f_\lambda$ by induction.
\begin{prop}
	Let $\mathfrak{g}$ be a finite dimensional complex simple Lie algebra and $\lambda \in \Lambda(\mathfrak{g})$ be the dominant weight of irreducible $\mathfrak{g}$-module $V(\lambda)$.
	Then the orbit factor $f_{\lambda}$ is a rational homogeneous polynomial and $\deg f_{\lambda}=\lvert O_{\lambda} \rvert$.	
\end{prop}
\begin{proof}
	Let $V(\lambda)$ be the irreducible module for classical Lie algebra $\mathfrak{g}$ with highest weight $\lambda$ for $\mu$ being domaint weight. We have
	$$ch_\lambda=\sum_{O_\mu \in orb(V(\lambda))}m_\lambda(\mu)\left(\sum_{\mu'\in O_\mu}e(\mu')\right).$$
	It is known that $m_\lambda(\lambda)=1$.
	
	From Remark \ref{3.4} and Theorem \ref{3.5}, and let
	$$g_{\lambda}=\prod_{\stackrel{\mu \prec \lambda}{\mu\in orb(V(\lambda))}}(f_\mu)^{m_{\lambda}(\mu)}.$$
	Then we have
	$$f_\lambda=\frac{f_{V(\lambda)}}{g_{\lambda}}.$$
	
	By induction, $f_\mu$ is a rational homogeneous polynomial with $\mu \prec \lambda$ and we see that
	$$\sum_{\stackrel{\mu \prec \lambda}{\mu\in orb(V(\lambda))}}\deg f_\mu=\sum_{\stackrel{\mu \prec \lambda}{\mu\in orb(V(\lambda))}}\lvert O_{\mu} \rvert m_\mu, $$ 
	and $g_{\lambda}$ is also a homogeneous polynomial. 
	Therefore $f_\lambda$ is a homogeneous polynomial
	with $\deg f_\lambda=\deg f_{V(\lambda)} - \deg f_{\lambda}=\lvert O_{\lambda} \rvert.$
\end{proof}

\subsection{Some results for $\mathfrak{sl}(4, \C)$}
\hspace{1.3em}We recall that Cartan subalgebra $\mathfrak{h}$ of $\mathfrak{sl}(4, \C)$, %
has a basis $\{H_1,H_2,$ $H_3\}$.
From Cartan matrix of root system $\Phi(\ddA_3)$ for Lie algebra $\mathfrak{sl}(4, \C)$, integral weight lattice $\Lambda(\ddA_3)$ has the integral basis
$$\beta_1=\frac{3\alpha_1+2\alpha_2+\alpha_3}{4},\ \beta_2=\frac{\alpha_1+2\alpha_2+\alpha_3}{2},\ \beta_3=\frac{\alpha_1+2\alpha_2+3\alpha_3}{4},$$
where $\alpha_1$, $\alpha_2$, $\alpha_3$ are the simple roots of $\Phi(\ddA_3)$.
It's well know that the Weyl group $\mathscr{W}(\ddA_3)$ is genetated by the reflections $\gamma_i (1\leq i \leq 3)$ and  have the presentation
$$\{\gamma_i, 1\leq i \leq 3 \mid \gamma_i^2=id,(\gamma_i\gamma_{i+1})^3=id, (\gamma_i\gamma_j)^2=id\ for\ \lvert i-j \rvert>1\}.$$

Let $\mu=l_1\beta_1+l_2\beta_2+l_3\beta_3$ be a dominant weight in $\Lambda(\ddA_3)$ with $l_1, l_2, l_3 \in \N$.
We obtain orbit
\begin{eqnarray*}
	O_{\mu}=\{\gamma(\mu)\mid \gamma\in\mathscr{W}(\ddA_3)\}.
\end{eqnarray*}
Let $L \in \mathfrak{sl}(4, \C)$ and have the form
\begin{equation*}
	L=\left(\begin{array}{cccc}
		z_{11} & z_{12} & z_{13} & z_{14}\\
		z_{21} & z_{22}-z_{11}   & z_{23}&z_{24}\\
		z_{31} & z_{32} & z_{33}-z_{22}  & z_{34} \\
		z_{41} & z_{42} & z_{43} & -z_{33} \\
	\end{array}\right),
\end{equation*}
with eigenvalues $\theta_1, \theta_2, \theta_3, \theta_4$, so
$$L\sim diag(\theta_1, \theta_2, \theta_3, \theta_4)=\theta_1H_1+(\theta_1+\theta_2)H_2+(\theta_1+\theta_2+ \theta_3)H_3.$$
Let $f_{\phi}(z_0,L)$ be the characteristic polynomial of $\mathfrak{sl}(4, \C)$ with respect to the representation $\phi$.
From Section \ref{An},
we have orbital factor
\begin{eqnarray*}
	f_{\mu}=\prod_{\gamma(\mu)\in O_{\mu}}\big(z_0+\theta_1 \gamma(\mu)(H_1)+(\theta_1+\theta_2) \gamma(\mu)(H_2)+(\theta_1+\theta_2+\theta_3)\gamma(\mu)(H_3)\big).
\end{eqnarray*}
Let
$$v=\theta_1 \gamma(\mu)(H_1)+(\theta_1+\theta_2) \gamma(\mu)(H_2)+(\theta_1+\theta_2+\theta_3)\gamma(\mu)(H_3).$$
In brief, we list the orbit $O_{\mu}$ and values for $v$  in the table below.

\begin{center}
	\begin{table}[htbp]
		\footnotesize
		\caption{weights in $O_{\mu}$ and the values for $v$}
		\centering
		\renewcommand\arraystretch{1.0}{
			\begin{tabular}{|p{6.2cm}<{\centering}|p{6.2cm}<{\centering}|}
			\Xhline{1.5px} 
			$Id$ & $\gamma_1$ \\
			\Xhline{1px} 
			$l_1\beta_1+l_2\beta_2+l_3\beta_3$ & $-l_1\beta_1+(l_1+l_2)\beta_2+l_3\beta_3$   \\
			\Xhline{1px} 
			$\theta_1l_1+(\theta_1+\theta_2)l_2+(\theta_1+\theta_2+\theta_3)l_3$ & $\theta_2l_1+(\theta_1+\theta_2)l_2+(\theta_1+\theta_2+\theta_3)l_3$ \\

			\Xhline{1.5px} 
			$\gamma_2$ & $\gamma_3$ \\
			\Xhline{1px}
			$(l_1+l_2)\beta_1-l_2\beta_2+(l_2+l_3)\beta_3$ & $l_1\beta_1+(l_2+l_3)\beta_2-l_3\beta_3$   \\
			\Xhline{1px} 
			$\theta_1l_1+(\theta_1+\theta_3)l_2+(\theta_1+\theta_2+\theta_3)l_3$ &
			$\theta_1l_1+(\theta_1+\theta_2)l_2+(\theta_1+\theta_2+\theta_4)l_3$ \\	
			
			\Xhline{1.5px}
			$\gamma_1\gamma_2\gamma_1$ & $\gamma_2\gamma_3\gamma_2$ \\
			\Xhline{1px}
			$-l_2\beta_1-l_1\beta_2+(l_1+l_2+l_3)\beta_3$ & $(l_1+l_2+l_3)\beta_1-l_3\beta_2-l_2\beta_3$   \\
			\Xhline{1px} 
			$\theta_3l_1+(\theta_2+\theta_3)l_2+(\theta_1+\theta_2+\theta_3)l_3$ &
			$\theta_1l_1+(\theta_1+\theta_4)l_2+(\theta_1+\theta_3+\theta_4)l_3$ \\	
			
			\Xhline{1.5px}
			$\gamma_3\gamma_2\gamma_1\gamma_2\gamma_3$ & $\gamma_1\gamma_3$ \\
			\Xhline{1px}
			$-(l_2+l_3)\beta_1+l_2\beta_2-(l_1+l_3)\beta_3$ & $-l_1\beta_1+(l_1+l_2+l_3)\beta_2-l_3\beta_3$   \\
			\Xhline{1px} 
			$\theta_4l_1+(\theta_2+\theta_4)l_2+(\theta_2+\theta_3+\theta_4)l_3$ &
			$\theta_2l_1+(\theta_1+\theta_2)l_2+(\theta_1+\theta_2+\theta_4)l_3$ \\	
			
			\Xhline{1.5px}
			$\gamma_2\gamma_1\gamma_3\gamma_2$ & $\gamma_3\gamma_2\gamma_1\gamma_2\gamma_3\gamma_2$ \\
			\Xhline{1px}
			$l_3\beta_1-(l_1+l_2+l_3)\beta_2+l_1\beta_3$ & $-l_3\beta_1-l_2\beta_2-l_1\beta_3$   \\
			\Xhline{1px} 
			$\theta_3l_1+(\theta_3+\theta_4)l_2+(\theta_1+\theta_3+\theta_4)l_3$ &
			$\theta_4l_1+(\theta_3+\theta_4)l_2+(\theta_2+\theta_3+\theta_3)l_4$ \\	
			
			\Xhline{1.5px}
			$\gamma_1\gamma_2\gamma_3$ & $\gamma_1\gamma_2\gamma_3\gamma_1\gamma_2$ \\
			\Xhline{1px}
			$-(l_1+l_2+l_3)\beta_1+l_1\beta_2+l_2\beta_3$ & $-l_3\beta_1-(l_1+l_2)\beta_2+l_1\beta_3$   \\
			\Xhline{1px} 
			$\theta_2l_1+(\theta_2+\theta_3)l_2+(\theta_2+\theta_3+\theta_4)l_3$ &
			$\theta_3l_1+(\theta_3+\theta_4)l_2+(\theta_2+\theta_3+\theta_4)l_3$ \\	
			
			\Xhline{1.5px}
			$\gamma_2\gamma_1\gamma_3$ & $\gamma_1\gamma_3\gamma_2$ \\
			\Xhline{1px}
			$(l_2+l_3)\beta_1-(l_1+l_2+l_3)\beta_2+(l_1+l_2)\beta_3$ & $-(l_1+l_2)\beta_1+(l_1+l_2+l_3)\beta_2-(l_2+l_3)\beta_3$   \\
			\Xhline{1px} 
			$\theta_3l_1+(\theta_1+\theta_3)l_2+(\theta_1+\theta_3+\theta_4)l_3$ &
			$\theta_2l_1+(\theta_2+\theta_4)l_2+(\theta_1+\theta_2+\theta_4)l_3$ \\	
			
			\Xhline{1.5px}
			$\gamma_2\gamma_1\gamma_3\gamma_2\gamma_1$ & $\gamma_3\gamma_2\gamma_1$ \\
			\Xhline{1px}
			$l_3\beta_1-(l_2+l_3)\beta_2-l_1\beta_3$ & $l_2\beta_1+l_3\beta_2-(l_1+l_2+l_3)\beta_3$   \\
			\Xhline{1px} 
			$\theta_4l_1+(\theta_3+\theta_4)l_2+(\theta_1+\theta_3+\theta_4)l_3$ &
			$\theta_4l_1+(\theta_1+\theta_4)l_2+(\theta_1+\theta_2+\theta_4)l_3$ \\	
			
			\Xhline{1.5px}
			$\gamma_1\gamma_2$ & $\gamma_2\gamma_3$ \\
			\Xhline{1px}
			$-(l_1+l_2)\beta_1+l_1\beta_2+(l_2+l_3)\beta_3$ & $(l_1+l_2+l_3)\beta_1-(l_2+l_3)\beta_2+l_2\beta_3$   \\
			\Xhline{1px} 
			$\theta_2l_1+(\theta_2+\theta_3)l_2+(\theta_1+\theta_2+\theta_3)l_3$ &
			$\theta_1l_1+(\theta_1+\theta_3)l_2+(\theta_1+\theta_3+\theta_3)l_4$ \\	
			
			\Xhline{1.5px}
			$\gamma_1\gamma_2\gamma_3\gamma_2$ & $\gamma_1\gamma_2\gamma_1\gamma_3$ \\
			\Xhline{1px}
			$-(l_1+l_2+l_3)\beta_1+(l_1+l_2)\beta_2-l_2\beta_3$ & $-(l_2+l_3)\beta_1-l_1\beta_2+(l_1+l_2)\beta_3$   \\
			\Xhline{1px} 
			$\theta_2l_1+(\theta_2+\theta_4)l_2+(\theta_2+\theta_3+\theta_4)l_3$ &
			$\theta_3l_1+(\theta_2+\theta_3)l_2+(\theta_2+\theta_3+\theta_4)l_3$ \\	
			
			\Xhline{1.5px}
			$\gamma_2\gamma_1$ & $\gamma_3\gamma_1\gamma_2\gamma_1$ \\
			\Xhline{1px}
			$l_2\beta_1-(l_1+l_2)\beta_2+(l_1+l_2+l_3)\beta_3$ & $-l_2\beta_1+(l_2+l_3)\beta_2-(l_1+l_2+l_3)\beta_3$   \\
			\Xhline{1px} 
			$\theta_3l_1+(\theta_1+\theta_3)l_2+(\theta_1+\theta_2+\theta_3)l_3$ &
			$\theta_4l_1+(\theta_2+\theta_4)l_2+(\theta_1+\theta_2+\theta_4)l_3$ \\	
			
			\Xhline{1.5px}
			$\gamma_2\gamma_3\gamma_2\gamma_1$ & $\gamma_3\gamma_2$ \\
			\Xhline{1px}
			$(l_2+l_3)\beta_1-l_3\beta_2-(l_1+l_2)\beta_3$ & $(l_1+l_2)\beta_1+l_3\beta_2-(l_2+l_3)\beta_3$   \\
			\Xhline{1px} 
			$\theta_4l_1+(\theta_1+\theta_4)l_2+(\theta_1+\theta_3+\theta_4)l_3$ &
			$\theta_1l_1+(\theta_1+\theta_4)l_2+(\theta_1+\theta_2+\theta_4)l_3$ \\	
			\Xhline{1.5px}
	\end{tabular}}
\end{table}
\end{center}

From the above table we can obtain $O_{l_1\beta_1}$, $O_{l_2\beta_2}$, $O_{l_3\beta_3}$, $O_{l_1\beta_1+l_2\beta_2}$, $O_{l_1\beta_1+l_3\beta_3}$, $O_{l_2\beta_2+l_3\beta_3}$, $O_{l_1\beta_1+l_2\beta_2+l_3\beta_3}$ and the orbit factors $f_\mu$. 
For example, take $l_2=l_3=0$, we get $O_{l_1\beta_1}$.
Let's take $\mu=l_2\beta_2$,
by the Vieta's theorem, we have
\begin{eqnarray*}
	f_{\mu}&=&\prod_{\gamma(\mu)\in O_{l_2\beta_2}}\big(z_0+\theta_1 \gamma(\mu)(H_1)+(\theta_1+\theta_2) \gamma(\mu)(H_2)+(\theta_1+\theta_2+\theta_3)\gamma(\mu)(H_3)\big)\\
	&=&z_0^6+2pl_2^2z_0^4+(p^2l_2^4-4rl_2^4)z_0^2-q^2l_2^6,
\end{eqnarray*}
where
$$p=\theta_1\theta_2+\theta_1\theta_3+\theta_1\theta_4+\theta_2\theta_3+\theta_2\theta_4+\theta_3\theta_4,$$
$$q=\theta_1\theta_2\theta_3+\theta_1\theta_2\theta_4+\theta_1\theta_3\theta_4+\theta_2\theta_3\theta_4,$$
$$r=\theta_1\theta_2\theta_3\theta_4.$$
Therefore, $f_{\mu}\in \mathbb{C}[z_{11}, z_{22},z_{33}, z_{12},\dots, z_{43}].$

\section{Characteristic polynomials for $\ddB_{n}$, $\ddC_{n}$, $\ddD_{n}(n\geq 4)$ and $\ddG_2$ }\label{classical}
	
\hspace{1.3em}In this chapter, we mainly study the characteristic polynomials on the complex simple Lie algebra of type $\ddB_{n}$, $\ddC_{n}$, $\ddD_{n}(n\geq 4)$ namely $\mathfrak{o}(2n+1, \C)$, $\mathfrak{sp}(2n, \C)$, $\mathfrak{o}(2n, \C)(n\geq 2)$ as defined in \cite[Chapter 1]{H1972} and the exceptional Lie algebra $\ddG_2$.


\subsection{Characteristic polynomials for $\ddB_{n}$, $\ddC_{n}$, $\ddD_{n}(n\geq 4)$}

\begin{lemma}\label{4.1}
	For the 
	classical simple Lie algebra of $\mathfrak{o}(2n+1, \C)$, $\mathfrak{sp}(2n, \C)$, $\mathfrak{o}(2n, \C) (n \geq 2)$, each $L$ in them.
	If $\theta$ is an eigenvalue of $L$, then $-\theta$ is also an eigenvalue of $L$ and they have the same multiplicities.
	For $\mathfrak{o}(2n+1, \C)$, the characteristic polynomial $\mathrm{det}(z_0 I-L)$ have only odd powers of $z_0$. For $\mathfrak{sp}(2n, \C)$ and $\mathfrak{o}(2n, \C) (n \geq 2)$, the characteristic polynomial $\mathrm{det}(z_0 I-L)$ have only even powers of $z_0$.
\end{lemma}	
\begin{proof}
	We only clarify type $\mathfrak{sp}(2n, \C)$, other types are similar.
	Let
	\begin{equation*}
		J=\left(
		\begin{array}{cc}
			0 & I_n  \\
			-I_n &  0 \\
		\end{array}.
		\right)	
	\end{equation*}
	We know that
	$\mathfrak{sp}(2n, \C)$ consist of all $2n\times 2n$ matrices $L$ satisfying
	$$LJ+JL^T=0,$$
	where $L^T$ denotes the transpose matrix.
	This means that $L \sim -L^T$, and since $L \sim L^T$, we have $L \sim -L$. Therefore the conclusion is confirmed.
\end{proof}	

Let $E_{i,j}$ be the  matrix  with $1$ at the entry $(i,j)$, others being $0$.
The Chevalley bases for
the Cartan subalgebras  $\mathfrak{h}$ of $\mathfrak{o}(2n+1, \C)$, $\mathfrak{sp}(2n, \C)$, $\mathfrak{o}(2n, \C) (n \geq 2)$ are listed below respectively.\\
$\mathfrak{o}(2n+1, \C)(n \geq 1)$:\\
$H_i=E_{i+1,i+1}-E_{i+2,i+2}-E_{n+i+1,n+i+1}+E_{n+i+2,n+i+2} (1\leq i\leq n-1),\\ H_n=2E_{n+1,n+1}-2E_{2n+1,2n+1}$.\\
$\mathfrak{sp}(2n, \C)(n \geq 1)$:\\
$H_i=E_{i,i}-E_{i+1,i+1}-E_{n+i,n+i}+E_{n+i+1,n+i+1}) (1\leq i\leq n-1),\\
H_n=E_{n,n}-E_{2n,2n}$.\\
$\mathfrak{o}(2n, \C) (n \geq 2)$:\\
$H_i=E_{i,i}-E_{i+1,i+1}-E_{n+i,n+i}+E_{n+i+1,n+i+1} (1\leq i\leq n-1),\\
H_n=E_{n-1,n-1}+E_{n,n}-E_{2n-1,2n-1}+E_{2n,2n}$.

\begin{thm}\label{4.9}
	Let $\mathfrak{h}$ be the Cartan subalgebra of classical Lie algebra $\mathfrak{o}(2n+1, \C)$ with the canonical basis $\{H_i\}_{i=1}^{n}$ and $L\in\mathfrak{o}(2n+1, \C)$ can be diagonalizable with eigenvalues $0, \theta_1,\cdots,\theta_n, $
	$-\theta_1,\cdots,-\theta_n$, then
	$$L\sim \sum_{j=1}^{n-1}\bigg(\sum_{i=j}^{n-1}H_i+\frac{H_n}{2}\bigg)\theta_j+\frac{H_n}{2}\theta_n=\sum_{i=1}^{n-1}\bigg(\sum_{j=1}^{i}\theta_j\bigg)H_i+\frac{\sum_{j=1}^n\theta_j}{2}H_n.$$
\end{thm}

\begin{proof}
	We know that
	$$diag(0, \theta_1,\cdots,\theta_n,-\theta_1,\cdots,-\theta_n)=\sum_{i=1}^n\theta_i(E_{i+1,i+1}-E_{n+i+1,n+i+1}),$$
	with
	$$(E_{i+1,i+1}-E_{n+i+1,n+i+1})=\sum_{i}^{n-1}H_i+\frac{H_n}{2}(0<i<n),\ E_{n+1,n+1}-E_{2n+1,2n+1}=\frac{H_n}{2}.$$
	Therefore we have
	$$L\sim \sum_{j=1}^{n-1}\bigg(\sum_{i=j}^{n-1}H_i+\frac{H_n}{2}\bigg)\theta_j+\frac{H_n}{2}\theta_n=\sum_{i=1}^{n-1}\bigg(\sum_{j=1}^{i}\theta_j\bigg)H_i+\frac{\sum_{j=1}^n\theta_j}{2}H_n.$$
\end{proof}
Let $f_{\phi}(z_0,L)$ be the characteristic polynomial of $\mathfrak{o}(2n+1, \C)$ with respect to the representation $\phi$.
Then,
\begin{eqnarray*}
	f_{\phi}(z_0,L)&=&\det\left(z_0I+\phi\bigg(\sum_{i=1}^{n-1}\bigg(\sum_{j=1}^{i}\theta_j\bigg)H_i+\frac{\sum_{j=1}^n\theta_j}{2}H_n\bigg)\right)\\
	&=&\det\left(z_0I+\sum_{i=1}^{n-1}\bigg(\sum_{j=1}^{i}\theta_j\bigg)\phi(H_i)+\frac{\sum_{j=1}^n\theta_j}{2}\phi(H_n)\right)\\
	&=&\prod_{O_{\mu}\subset orb(V(\lambda))} (f_{\mu})^{m_\lambda(\mu)}
\end{eqnarray*}	
with orbital factor
\begin{eqnarray}\label{form}
	f_{\mu}=\prod_{\gamma(\mu)\in O_{\mu}}\left(z_0I+\sum_{i=1}^{n-1}\bigg(\sum_{j=1}^{i}\theta_j\bigg)\gamma(\mu)(H_i)+\frac{\sum_{j=1}^n\theta_j}{2}\gamma(\mu)(H_n)\right).
\end{eqnarray}

\begin{rem}\label{weyl}
For type $\ddB_n$ and $\ddC_n$, the Weyl group $\mathscr{W}(\ddB_n)=\mathscr{W}(\ddC_n)=\mathfrak{S}_n \ltimes \mathbb{Z}_2^n$ where $\mathbb{Z}_2^n=\{[\epsilon_1,\dots,\epsilon_n] \mid \epsilon_i=\pm1\}$. There is a representation of $\mathscr{W}(\ddB_n)$ on $\mathbb{C}^n$ where $\mathfrak{S}_n$ acts as permutations of coordinates and $\epsilon \in \mathbb{Z}_2^n$ acts by $\epsilon  (\theta_1,\dots,\theta_n)=[\epsilon_1\theta_1,\dots,\epsilon_n\theta_n].$
It's show that $\mathbb{C}[\theta_1,\dots,\theta_{n}]^{\mathfrak{S}_n \ltimes \mathbb{Z}_2^n}=\mathbb{C}[\bar{s}_1,\dots, \bar{s}_n]$ with
$$\bar{s}_i=\sum_{1\leq j_1 < \cdots < j_i \leq n}\theta_{j_1}^2 \cdots \theta_{j_i}^2.$$
For type $\ddD_n$, the Weyl group	$\mathscr{W}(\ddD_n)=\mathfrak{S}_n \ltimes \mathbb{Z}_2^{n-1}$
act on $\mathbb{C}^n$ by the rule that $\mathfrak{S}_n$ permutes coordinates, whereas $\mathbb{Z}_2^{n-1}$  consists of the elements changing the sign on an even number of coordinates, then $\mathbb{C}[\theta_1,\dots,\theta_{n}]^{\mathfrak{S}_n \ltimes \mathbb{Z}_2^{n-1}}=\mathbb{C}[\bar{s}_1,\dots, \bar{s}_{n-1},t]$ where $\bar{s}_i$ is as above and $t=\theta_1\theta_2\cdots\theta_n.$
For the above facts, we refer the reader to reference \cite[Section 16-2]{RK2001}.
\end{rem}

\begin{thm}\label{4.4}
	Extending the $\mathscr{W}(\ddB_{n})\simeq\mathfrak{S}_n\ltimes \mathbb{Z}_2^{n}$ actions on $z_{1,1},\dots, z_{n, n},\dots, z_{i,j},\dots$ to $z_0$ by trivial action.
	The orbit factor $f_{\mu}$ for $\mathfrak{o}(2n+1, \C)$ is a  $\mathscr{W}(\ddB_{n})$ invariant polynomial, 
	so it follows that $f_{\mu} \in \mathbb{C}[\bar{s}_1,\dots, \bar{s}_n].$
	Furthermore, given that $f_\mu$ is a polynomial of $z_0$, the coefficients of $f_\mu$ can be expressed by polynomials of $\bar{s}_1,\dots, \bar{s}_n$.
	And then $f_{\mu}$ is a factor of $f_\phi$.
\end{thm}

\begin{proof}
	$\forall \tau=\gamma\ltimes\epsilon \in \mathscr{W}(\ddB_n)$ with $\gamma \in \mathfrak{S}_n$, there is the following fact,
	
	\begin{eqnarray*}
		&&\tau (\mu)\left(\sum_{i=1}^{n-1}\bigg(\sum_{j=1}^{i}\theta_j\bigg)H_i+\frac{\sum_{j=1}^n\theta_j}{2}H_n\right)\\
		&=&\left(\sum_{k=1}^{n}l_k\tau(\beta_k)\right)  \left(\sum_{i=1}^{n-1}\bigg(\sum_{j=1}^{i}\theta_j\bigg)H_i+\frac{\sum_{j=1}^n\theta_j}{2}H_n\right)\\
		&=&\mu \left(\tau^{-1}\bigg(\sum_{i=1}^{n-1}\bigg(\sum_{j=1}^{i}\theta_j\bigg)H_i+\frac{\sum_{j=1}^n\theta_j}{2}H_n\bigg)\right)\\
		&=&\mu\left(\tau^{-1}(diag(0, \theta_1, \cdots, \theta_{n}, -\theta_1, \cdots, -\theta_{n} )) \right) \\
		&=&\mu\left(diag(0, \epsilon_1\theta_{\gamma(1)},  \cdots , \epsilon_n\theta_{\gamma(n)}, -\epsilon_1\theta_{\gamma(1)},  \cdots , -\epsilon_n\theta_{\gamma(n)}) \right) \\
		&=&\sum_{i=1}^{n-1}\bigg(\sum_{j=1}^{i}\epsilon_j\theta_{\gamma(j)}\bigg)\mu(H_i)+\frac{\sum_{j=1}^n\epsilon_j\theta_{\gamma(j)}}{2}\mu(H_n).\\
	\end{eqnarray*}
	
	
	Then we have
	\begin{eqnarray*}
		f_{\mu}&=&\prod_{\gamma(\mu)\in O_{\mu}}\left(z_0I+\sum_{i=1}^{n-1}\bigg(\sum_{j=1}^{i}\theta_j\bigg)\gamma(\mu)(H_i)+\frac{\sum_{j=1}^n\theta_j}{2}\gamma(\mu)(H_n)\right)\\
		&=&\prod_{\tau \in \mathscr{W}(\ddB_n)/stab(\mu)}\left(z_0+\tau(\mu)\bigg(\sum_{i=1}^{n-1}\bigg(\sum_{j=1}^{i}\theta_j\bigg)H_i+\frac{\sum_{j=1}^n\theta_j}{2}H_n\bigg)\right)\\
		&=&\prod_{\gamma\ltimes\epsilon \in \gamma \in \mathscr{W}(\ddB_n)/stab(\mu)}\left(z_0+\sum_{i=1}^{n-1}\bigg(\sum_{j=1}^{i}\epsilon_j\theta_{\gamma(j)}\bigg)\mu(H_i)+\frac{\sum_{j=1}^n\epsilon_j\theta_{\gamma(j)}}{2}\mu(H_n)\right) \\
		&=&\prod_{\gamma\ltimes\epsilon \in \gamma \in \mathscr{W}(\ddB_n)/stab(\mu)}\left(z_0+\sum_{i=1}^{n-1}\bigg(\sum_{j=1}^{i}\epsilon_j\theta_{\gamma(j)}\bigg)l_i+\frac{\sum_{j=1}^n\epsilon_j\theta_{\gamma(j)}}{2}l_n\right) \\
		&=&\prod_{\gamma\ltimes\epsilon \in \gamma \in \mathscr{W}(\ddB_n)/stab(\mu)}\left(z_0^2-\bigg(\sum_{i=1}^{n-1}\bigg(\sum_{j=1}^{i}\epsilon_j\theta_{\gamma(j)}\bigg)l_i+\frac{\sum_{j=1}^{n-1}\epsilon_j\theta_{\gamma(j)}+\theta_{n}}{2}l_n\bigg)^2\right).
	\end{eqnarray*}	
	So we know that the coefficient of $f_{\mu}$ as the polynomial of $z_0$ is a $\mathscr{W}(\ddB_n)$ invariant polynomial and can be written uniquely as an algebraic expression in $\{\bar{s}_1, \dots, \bar{s}_{n}\}$ from \cite[Section 16-2]{RK2001}.
\end{proof}

Through the same argument, there are also the following two theorems.

\begin{thm}\label{4.10}
	Let $\mathfrak{h}$ be the Cartan subalgebra of classical Lie algebra  $\mathfrak{sp}(2n, \C)$  with the canonical basis $\{H_i\}_{i=1}^{n}$ and $L\in\mathfrak{sp}(2n, \C)$  can be diagonalizable with eigenvalues $ \theta_1,\cdots,\theta_n, -\theta_1,\cdots,$
	$-\theta_n,$ then
	$$L\sim \sum_{j=1}^{n}\bigg(\sum_{i=j}^{n}H_i\bigg)\theta_j=\sum_{i=1}^{n}\bigg(\sum_{j=1}^{i}\theta_j\bigg)H_i.$$
\end{thm}

\begin{thm}\label{4.11}
	Let $\mathfrak{h}$ be the Cartan subalgebra of classical Lie algebra $\mathfrak{o}(2n, \C) (n \geq 2)$ with the canonical basis $\{H_i\}_{i=1}^{n}$ and $L\in\mathfrak{o}(2n, \C)$  can be diagonalizable with eigenvalues $ \theta_1,\cdots,\theta_n, $
	$-\theta_1,\cdots,-\theta_n,$ then
	\begin{eqnarray*}
		L&\sim& \sum_{j=1}^{n-2}\bigg(\sum_{i=j}^{n-2}H_i+\frac{H_{n-1}+H_n}{2}\bigg)\theta_j+\frac{H_{n-1}+H_n}{2}\theta_{n-1}+\frac{H_n-H_{n-1}}{2}\theta_n\\
		&=&\sum_{i=1}^{n-2}\bigg(\sum_{j=1}^{i}\theta_j\bigg)H_i+\frac{\sum_{j=1}^{n-1}\theta_j-\theta_n}{2}H_{n-1}+\frac{\sum_{j=1}^{n}\theta_j}{2}H_n.
	\end{eqnarray*}
\end{thm}

\begin{thm}\label{4.13}
	Extending the $\mathscr{W}(\ddC_{n})\simeq\mathfrak{S}_n\ltimes \mathbb{Z}_2^{n}$ actions on $z_{1,1},\dots, z_{n, n}, \dots, z_{i,j}, \dots$ to $z_0$ by trivial action.
	The orbit factor $f_{\mu}$ for $\mathfrak{sp}(2n, \C)$ is a  $\mathscr{W}(\ddC_{n})$ invariant polynomial, 
	so it follows that $f_{\mu} \in \mathbb{C}[\bar{s}_1,\dots, \bar{s}_n].$
	Furthermore, given that $f_\mu$ is a polynomial of $z_0$, the coefficients of $f_\mu$ can be expressed by polynomials of $\bar{s}_1,\dots, \bar{s}_n$.
	And then $f_{\mu}$ is a factor of $f_\phi$.
\end{thm}

\begin{thm}\label{4.14}
	Extending the $\mathscr{W}(\ddD_{n})\simeq\mathfrak{S}_n\ltimes \mathbb{Z}_2^{n-1}$ actions on $z_{1,1},\dots, z_{n, n},\dots, z_{i,j}, \dots$ to $z_0$ by trivial action.
	The orbit factor $f_{\mu}$ for $\mathfrak{sp}(2n, \C)$ is a  $\mathscr{W}(\ddD_{n})$ invariant polynomial, 
	so it follows that $f_{\mu} \in \mathbb{C}[\bar{s}_1,\dots, \bar{s}_{n-1},t].$
	Furthermore, given that $f_\mu$ is a polynomial of $z_0$, the coefficients of $f_\mu$ can be expressed by polynomials of $\bar{s}_1,\dots, \bar{s}_{n-1},t$.
	And then $f_{\mu}$ is a factor of $f_\phi$.
\end{thm}

From Theorem \ref{irreducible} and the above conclusions, the theorem below holds.
\begin{thm}\label{4.15}
Let $\mathfrak{g}$ be the classical Lie algebra with a Chevalley basis, $\mu_1, \dots, \mu_s$ be the different dominant weights of irreducible $\mathfrak{g}$-module $V$ and $k_i$ be the size of the orbit of $\mu_i$ under the Weyl group  $\mathscr{W}$, $i=1, \dots, s$.
Then the characteristic polynomial
$$f_{\phi}(z_0,L)=(f_{\mu_1})^{m_\lambda(\mu_1)}\cdots (f_{\mu_s})^{m_\lambda(\mu_s)},$$
with $f_{\mu_i}$ irreducible of degree $k_i=\lvert O_{\mu_i} \rvert$ obtained from the orbit of $\mu_i$.
\end{thm}

\begin{example}
For the $B_2$ with a basis of Cartan subalgebra $\{H_1=diag(0, 1, -1, -1, 1),\ H_2=diag(0, 0, 2, 0, -2)\}$, it follows that
\begin{eqnarray}O_{l_1l_2}=\begin{cases}
\{l_1\beta_1,\ -l_1\beta_1+2l_1\beta_2,\ -l_1\beta_1,\ l_1\beta_1-2l_1\beta_2\}, \quad l_1\neq 0, \quad l_2=0\\
\\
\{l_2\beta_2,\ l_2\beta_1-l_2\beta_2,\ -l_2\beta_2,\ -l_2\beta_1+l_2\beta_2\}, \quad l_2\neq 0, \quad l_1=0\\
\\
\{l_1\beta_1+l_2\beta_2,\ -l_1\beta_1+(2l_1+l_2)\beta_2,\ (l_1+l_2)\beta_1-(2l_1+l_2)\beta_2,\\
(l_1+l_2)\beta_1-l_2\beta_2,\
-l_1\beta_1-l_2\beta_2,\ l_1\beta_1-(2l_1+l_2)\beta_2,\\
-(l_1+l_2)\beta_1+(2l_1+l_2)\beta_2,\
-(l_1+l_2)\beta_1+l_2\beta_2\},\quad l_1l_2\neq 0\\

\\
\{0\}, \quad l_1=l_2=0
\end{cases}
\end{eqnarray}
From (\ref{form}),
all orbital factors are obtained.

\textbf{Case 1:}

When $\mu=0$, $f_{\mu}=z_0.$

\textbf{Case 2:}

When $\mu=l_1\beta_1$,
\begin{eqnarray*}f_{\mu}&=&z_0\big(z_0^2-(\theta_1l_1)^2\big)\big(z_0^2-(\theta_2l_1)^2\big) \\
	&=&z_0\big(z_0^4-l_1^2\bar{s}_1z_0^2+l_1^2l_2^2\bar{s}_2\big).
\end{eqnarray*}

\textbf{Case 3:}

When $\mu=l_2\beta_2$,
\begin{eqnarray*}f_{\mu}&=&z_0\big(z_0^2-(\frac{\theta_1+\theta_2}{2}l_2)^2\big)\big(z_0^2-(\frac{-\theta_1+\theta_2}{2}l_2)^2\big) \\
&=&z_0\big(z_0^4-\frac{1}{2}l_2^2\bar{s}_1z_0^2+\frac{1}{16}l_2^4(\bar{s}_1^2-2\bar{s}_2)-\frac{1}{8}l_2^4\bar{s}_2\big).
\end{eqnarray*}

\textbf{Case 4:}

When $\mu=l_1\beta_1+l_2\beta_2$,
\begin{eqnarray*}f_{\mu}
=&z_0\big(z_0^2-(\theta_1l_1+\frac{\theta_1+\theta_2}{2}l_2)^2\big)\big(z_0^2-(\theta_2l_1+\frac{\theta_1+\theta_2}{2}l_2)^2\big) &\\
	&\big(z_0^2-(-\theta_1l_1+\frac{-\theta_1+\theta_2}{2}l_2)^2\big)\big(z_0^2-(\theta_2l_1+\frac{-\theta_1+\theta_2}{2}l_2)^2\big)& \\
    &=z_0\big(z_0^8+q_{l_1,l_2}^1z_0^6+q_{l_1,l_2}^2z_0^4+q_{l_1,l_2}^3z_0^2+q_{l_1,l_2}^4\big),
\end{eqnarray*}
where
\begin{alignat*}{4}
		q_{l_1,l_2}^1&=(-2l_1^2-2l_1l_2-l_2^2)\bar{s}_1,\\
		q_{l_1,l_2}^2&=(l_1^4+2l_1^3 l_2+\frac{5}{2} l_1^2 l_2^2+\frac{3}{2} l_1 l_2^3+\frac{3}{8}l_2^4)(\bar{s}_1^2-2\bar{s}_2)
		+(4l_1^4+8l_1^3 l_2+5l_1^2 l_2^2+l_1 l_2^3+\frac{1}{4}l_2^4)\bar{s}_2,\\
			q_{l_1,l_2}^3&=-\frac{1}{16} l_2^2(2l_1+l_2)^2(2l_1^2+2l_1 l_2+l_2^2)(\bar{s}_1^3-3\bar{s}_1\bar{s}_2)\\
		&-\frac{1}{16}(4l_1^2+2l_1 l_2-l_2^2)(2l_1^2+2l_1 l_2+l_2^2)(4l_1^2+6l_1 l_2+l_2^2)\bar{s}_1\bar{s}_2,\\
		q_{l_1,l_2}^4&=\frac{1}{256}l_2^4(2l_1+l_2)^4(\bar{s}_1^4+2\bar{s}_2^2-4\bar{s}_1^2\bar{s}_2)\\
		&-\frac{1}{64}l_2^2(2l_1+l_2)^2 (8l_1^4+16l_1^3 l_2+12l_1^2 l_2^2+4l_1 l_2^3+l_2^4)\bar{s}_2(\bar{s}_1^2-2\bar{s}_2)\\
		&+(l_1^8+4l_1^7 l_2+7l_1^6 l_2^2+7l_1^5 l_2^3+\frac{37}{8} l_1^4 l_2^4+\frac{9}{4} l_1^3 l_2^5+\frac{13}{16} l_1^2 l_2^6+\frac{3}{16} l_1 l_2^7+\frac{3}{128}l_2^8)\bar{s}_2^2.
		\end{alignat*}

Therefore, the results are in agreement with the Theorem \ref{4.4}.

It is noteworthy that
$$z_0^2-(\frac{\theta_1+\theta_2}{2}l_2)^2,\ \big(z_0^2-(\theta_1l_1+\frac{\theta_1+\theta_2}{2}l_2)^2\big)\big(z_0^2-(\theta_2l_1+\frac{\theta_1+\theta_2}{2}l_2)^2\big)$$
are symmetric polynomial, but not  $\mathscr{W}(\ddB_{2})$ invariant polynomials, then not orbital factors of $f_{\phi}(z_0,L)$.
\end{example}

\subsection{Characteristic polynomials for $\ddG_{2}$}

\hspace{1.3em}By the same method, we try to compute the characteristic polynomial of
the exceptional Lie algebra $\ddG_2$ within Weyl group $\mathscr{D}_6 \simeq \mathfrak{S}_3 \ltimes \mathbb{Z}_2$. We take a Chevalley basis for Cartan algebra $\mathfrak{h}$ of $\ddG_2$ as follows
	\begin{equation*}
	H_1=\left(\begin{array}{ccccccc}
		0 & 0 & 0 & 0 & 0 & 0 & 0  \\
		0 & 1 & 0 & 0 & 0 & 0 & 0  \\
		0 & 0 & -1 & 0 & 0 & 0 & 0  \\
		0 & 0 & 0 & 2 & 0 & 0 & 0  \\
		0 & 0 & 0 & 0 & 1 & 0 & 0  \\
		0 & 0 & 0 & 0 & 0 & -1 & 0  \\
		0 & 0 & 0 & 0 & 0 & 0 & -2  \\
	\end{array}\right), \quad
    H_2=\left(\begin{array}{ccccccc}
		0 & 0 & 0 & 0 & 0 & 0 & 0  \\
		0 & -1 & 0 & 0 & 0 & 0 & 0  \\
		0 & 0 & 0 & 0 & 0 & 0 & 0  \\
		0 & 0 & 0 & -1 & 0 & 0 & 0  \\
		0 & 0 & 0 & 0 & 0 & 0 & 0  \\
		0 & 0 & 0 & 0 & 0 & 1 & 0  \\
		0 & 0 & 0 & 0 & 0 & 0 & 1  \\
	\end{array}\right).	
\end{equation*}
The other elements of the basis are not detailed and readers can refer to \cite[Chapter 4.4]{RA2005}.
In general, an element of $\ddG_2$ may be written in the form
\begin{equation*}
	L=\left(\begin{array}{ccccccc}
		0 & -2z_{12} & 2z_{13} & 2z_{14} & 2z_{15} & -2z_{16} & 2z_{17}  \\
		z_{16} & z_{11}-z_{22} & z_{17} & -z_{15} & z_{25} & 0 & z_{27}  \\
		z_{15} & z_{14} & -z_{11} & -z_{34} & 0 & z_{25} & -z_{16}  \\
		z_{17} & -z_{13} & z_{43} & 2z_{11}-z_{22} & z_{16} & z_{27} & 0  \\
		z_{13} & z_{52} & 0 & -z_{12} & z_{11} & z_{17} & -z_{43}  \\
		z_{12} & 0 & z_{52} & z_{64} & z_{14} & -z_{11}+z_{22} & -z_{13}  \\
		z_{14} & z_{64} & z_{12} & 0 & z_{34} & -z_{15} & -2z_{11}+z_{22} \\
	\end{array}\right)
\end{equation*}
where $z_{ij}$ are indeterminants.
Let $L$ can be diagonalizable and have eigenvalues $\{0, \theta_1, \theta_2, \theta_3, -\theta_1, $ $-\theta_2, -\theta_3\}$ and we know that $\theta_1+\theta_2+\theta_3=0$, so
$$L\sim diag(0, \theta_1, \theta_2, \theta_3, -\theta_1, -\theta_2, -\theta_3) \sim \theta_1H_1+(2\theta_1+\theta_2)H_2=(-\theta_2-\theta_3)H_1+(\theta_1-\theta_3)H_2.$$
Then characteristic polynomial
\begin{eqnarray*}
	f_{\phi}(z_0,L)
	&=&\det\left(z_0I+(-\theta_2-\theta_3)\phi(H_1)+(\theta_1-\theta_3)\phi(H_2)\right)\\
	&=&\prod_{O_{\mu}\subset orb(V(\lambda))} (f_{\mu})^{m_\lambda(\mu)}.
\end{eqnarray*}	
Next, $f_{\mu}$ will be shown to be an invariant polynomial.
Let $\mu=l_1\beta_1+l_2\beta_2$ be a dominant weight in $\Lambda(\ddG_2)$ with $l_1, l_2\in \N$.
$\forall \tau=\gamma\ltimes\epsilon \in \mathscr{D}_6$ with $\gamma \in \mathfrak{S}_3, \epsilon \in \mathbb{Z}_2$, there is the following fact,
	\begin{eqnarray*}
	&&\tau (\mu)\big((-\theta_2-\theta_3)H_1+(\theta_1-\theta_3)H_2\big)\\
	&=&\mu \left(\tau^{-1}\big((-\theta_2-\theta_3)H_1+(\theta_1-\theta_3)H_2\big)\right)\\
	&=&\mu\left(\tau^{-1}(diag(0, \theta_3,  -\theta_1, -\theta_{2}, \theta_1,  -\theta_3, \theta_{2} )) \right) \\
	&=&\mu\left(diag(0, \epsilon\theta_{\gamma(3)},  -\epsilon\theta_{\gamma(1)}, -\epsilon\theta_{\gamma(2)}, \epsilon\theta_{\gamma(1)}, -\epsilon\theta_{\gamma(3)}, \epsilon\theta_{\gamma(2)}) \right) \\
	&=&(-\epsilon\theta_{\gamma(2)}-\epsilon\theta_{\gamma(3)})\mu(H_1)+(\epsilon\theta_{\gamma(1)}-\epsilon\theta_{\gamma(3)})\mu(H_2).\\
\end{eqnarray*}
 Then
	\begin{eqnarray*}
	f_{\mu}&=&\prod_{\tau \in \mathscr{D}_6/stab(\mu)}\left(z_0+\tau(\mu)\big((-\theta_2-\theta_3)H_1+(\theta_1-\theta_3)H_2\big)\right)\\
	&=&\prod_{\tau \in \mathscr{D}_6/stab(\mu)}\left(z_0+(-\epsilon\theta_{\gamma(2)}-\epsilon\theta_{\gamma(3)})\mu(H_1)+(\epsilon\theta_{\gamma(1)}-\epsilon\theta_{\gamma(3)})\mu(H_2)\right)\\
	&=&\prod_{\tau \in \mathscr{D}_6/stab(\mu)}\left(z_0+(-\epsilon\theta_{\gamma(2)}-\epsilon\theta_{\gamma(3)})l_1+(\epsilon\theta_{\gamma(1)}-\epsilon\theta_{\gamma(3)})l_2\right)\\
	&=&\prod_{\gamma \in \mathscr{D}_6/stab(\mu)/\mathbb{Z}_2}\left(z_0^2-\big((-\theta_{\gamma(2)}-\theta_{\gamma(3)})l_1+(\theta_{\gamma(1)}-\theta_{\gamma(3)})l_2\big)^2\right).\\
\end{eqnarray*}
Therefore $f_{\mu}$ is a symmetric polynomial and a $\mathfrak{S}_3$ invariant polynomial and can be written uniquely as an algebraic expression in $\{\bar{s}_1, \bar{s}_{2}, \bar{s}_{3}\}$ from \cite[Section 16-2]{RK2001}.

In conclusion, the following theorem is obtained.
\begin{thm}\label{1}
	Let $\mu_1, \dots, \mu_s$ be the dominant weights of irreducible $\ddG_2$-module $V$ and $k_i$ be the size of the orbit of $\mu_i$ under the $\mathscr{W}(\ddG_2)$, $i=1, \dots, s$. Then $$f_{\phi}(z_0,L)=(f_{\mu_1})^{m_\lambda(\mu_1)}\cdots (f_{\mu_s})^{m_\lambda(\mu_s)},$$
	where $f_{\mu_i}$ is an irreducible polynomials of degree $k_i=\lvert O_{\mu_i} \rvert$ obtained from the orbit of $\mu_i$.
\end{thm}

	\begin{rem} From this paper, we see that the characteristic polynomials of the representations of a  complex simple Lie algebra have a profound meaning for the representation monoidal category of the Lie algebras.
		However, we have not presented the relevant conclusions of $F_4$, $E_6$, $E_7$, $E_8$, because the calculation method is the same as the classical Lie algebras, the computation might be more difficult for precise computing works.
		In fact, there are many interesting topics about these polynomials, such as finding the link between the coefficients  and the representations,
		how to factorize the tensor products through the resolution products of their characteristic polynomials and so on. Therefore, we need more efforts to work on these topics.
	\end{rem}

Chenyue Feng\\
Email: fcy 1572565124@163.com\\
School of Mathematics, Shandong University\\
Shanda Nanlu 27, Jinan, \\
Shandong Province, China\\
Postcode: 250100\\
Shoumin Liu\\
Email: s.liu@sdu.edu.cn\\
School of Mathematics, Shandong University\\
Shanda Nanlu 27, Jinan, \\
Shandong Province, China\\
Postcode: 250100\\
Xumin Wang\\
Email: 202320303@mail.sdu.edu.cn\\
School of Mathematics, Shandong University\\
Shanda Nanlu 27, Jinan, \\
Shandong Province, China\\
Postcode: 250100

\end{document}